\documentclass[11pt,twoside]{article}
\usepackage{amsmath,amssymb,amsthm}
\usepackage{graphics,graphicx}
\usepackage{epstopdf}
\usepackage{lipsum}
\usepackage{braket,amsfonts,amsopn}
\usepackage{subfigure}
\usepackage{algorithm}
\usepackage{algorithmic}
\usepackage{enumerate} 
\usepackage{cases}
\usepackage{color}
\usepackage{fancyhdr}
\usepackage{empheq}
\usepackage{marginnote}
\usepackage{optidef}
\usepackage{float}
\usepackage{multicol}
\usepackage{multirow}
\usepackage{hyperref}
\usepackage{cleveref}  
\usepackage{blkarray}
\usepackage{enumerate}
\usepackage{verbatim}
\usepackage{optidef}
\usepackage[asymmetric]{geometry}
\usepackage{cancel}

\ifpdf
\DeclareGraphicsExtensions{.eps,.pdf,.png,.jpg}
\else
\DeclareGraphicsExtensions{.eps}
\fi

\newcommand\AMSname{AMS subject classifications}

\newcommand\keywordsname{Key words}

\newcommand{\newcomment}[1]{}

\newenvironment{@abssec}[1]{%
	\if@twocolumn
	\section*{#1}%
	\else
	\vspace{.05in}\footnotesize
	\parindent .2in
	{\upshape\bfseries #1. }\ignorespaces
	\fi}
{\if@twocolumn\else\par\vspace{.1in}\fi}

\newenvironment{AMS}{\begin{@abssec}{\AMSname}}{\end{@abssec}}

\newtheorem{definition}{Definition}[section]

\numberwithin{equation}{section}
\numberwithin{figure}{section}
\numberwithin{table}{section}

\newtheorem{theorem}{Theorem}[section]

\newtheorem{lemma}{Lemma}[section]

\newtheorem{corollary}{Corollary}[section]
\newtheorem{remark}{Remark}[section]

\newcommand{\email}[1]{\protect\href{mailto:#1}{#1}}

\floatname{algorithm}{Algorithm}

\newenvironment{keywords}{\begin{@abssec}{\keywordsname}}{\end{@abssec}}

\usepackage{amsopn}

\DeclareMathOperator{\Real}{Re}
\DeclareMathOperator{\imag}{Imag}
\DeclareMathOperator{\Span}{span}
\DeclareMathOperator{\sign}{sign}
\DeclareMathOperator{\dist}{dist}
\DeclareMathOperator{\Diag}{Diag}

\DeclareMathOperator{\ii}{i}

\newcommand{\mat}[1]{\left[ \begin{array}{#1} }
\newcommand{\rix}{\end{array} \right]}

\fancypagestyle{plain}{
	\fancyhf{}
	\cfoot{\thepage}
	
}

\makeatletter
\def\bbordermatrix#1{\begingroup \m@th
	\@tempdima 4.75\p@
	\setbox\z@\vbox{%
		\def\cr{\crcr\noalign{\kern2\p@\global\let\cr\endline}}%
		\ialign{$##$\hfil\kern2\p@\kern\@tempdima&\thinspace\hfil$##$\hfil
			&&\quad\hfil$##$\hfil\crcr
			\omit\strut\hfil\crcr\noalign{\kern-\baselineskip}%
			#1\crcr\omit\strut\cr}}%
	\setbox\tw@\vbox{\unvcopy\z@\global\setbox\@ne\lastbox}%
	\setbox\tw@\hbox{\unhbox\@ne\unskip\global\setbox\@ne\lastbox}%
	\setbox\tw@\hbox{$\kern\wd\@ne\kern-\@tempdima\left[\kern-\wd\@ne
		\global\setbox\@ne\vbox{\box\@ne\kern2\p@}%
		\vcenter{\kern-\ht\@ne\unvbox\z@\kern-\baselineskip}\,\right]$}%
	\null\;\vbox{\kern\ht\@ne\box\tw@}\endgroup}
\makeatother

\title{2D Eigenvalue Problem III: Convergence Analysis of
the 2D Rayleigh Quotient Iteration~\thanks{
version dated \today.
}
}

\author{
Tianyi Lu\thanks{School of Mathematical Sciences, 
Fudan University, Shanghai 200433, China
(\email{tylu17@fudan.edu.cn}, \email{yfsu@fudan.edu.cn}). 
}
\and Yangfeng Su\footnotemark[2]
\and Zhaojun Bai\thanks{Department of Computer Science
and Department of Mathematics, University of California,
Davis, CA 95616, USA (\email{zbai@ucdavis.edu})}
}
\date{\today}

\ifpdf
\hypersetup{
pdftitle={2D Eigenvalue Problems III},
pdfauthor={T. Lu, Y. Su and Z. Bai}
}
\fi

\begin{document}
\maketitle

\begin{abstract}
In Part I of this paper, we introduced
a two dimensional eigenvalue problem (2DEVP) of a matrix pair 
and investigated its fundamental theory such as existence, 
variational characterization and number of 2D-eigenvalues.
In Part II, we proposed a Rayleigh quotient iteration (RQI)-like 
algorithm (2DRQI) for computing a 2D-eigentriplet of the 2DEVP 
near a prescribed point, and discussed applications of 2DEVP and
2DRQI for solving the minimax problem of Rayleigh quotients, 
and computing the distance to instability. 
In this third part, we present convergence analysis of the 2DRQI. 
We show that under some mild conditions, 
the 2DRQI is locally quadratically convergent for computing a nonsingular 
2D-eigentriplet. 
\end{abstract}




\section{Introduction}\label{eq:intro}
Given Hermitian matrices $A, C \in \mathbb{C}^{n\times n}$ and 
$C$ is indefinite, the 2D eigenvalue problem (2DEVP) is to 
find scalars $\mu, \lambda \in \mathbb{R}$ and nonzero vectors 
$x \in \mathbb{C}^n$ such that
\begin{subequations}\label{2deig}
\begin{empheq}[left={}\empheqlbrace]{alignat=2}
(A-\mu C) x & = \lambda x,  \label{eq:1a} \\
x^HCx & = 0, \label{eq:1b}   \\
x^Hx & = 1,  \label{eq:1c}
\end{empheq}
\end{subequations}
The pair $(\mu, \lambda)$ is called a \emph{2D-eigenvalue}, 
$x$ is called the corresponding \emph{2D-eigenvector}, and 
the triplet $(\mu,\lambda,x)$ is called a \emph{2D-eigentriplet}. 
We use the term ``$2D$'' based on the fact that an eigenvalue 
has two components,  which is a point in the
two dimensional $(\mu, \lambda)$-plane.

In Part I of this work~\cite{2DEVPI}, we presented fundamental properties 
of the 2DEVP such as the existence, variational characterizations, and 
the necessary and sufficient conditions for the finite number of 
2D-eigenvalues. We also discussed the applications of the 2DEVP 
on the minimax problem of Rayleigh quotients and the computation 
of distance to instability.
In Part II \cite{2DEVPII}, we proposed a Rayleigh quotient iteration 
(RQI)-like algorithm (2DRQI) for computing an eigentriplet of the 2DEVP 
near a prescribed point. Numerical experiments show its promising 
performance compared with eigenvalue optimization algorithms
for finding the minimax of Rayleigh quotients and computing distance 
to instability.

In this part, we present the convergence analysis of the 2DRQI.
We will prove that the 2DRQI is locally quadratically convergent for
computing a nonsingular 2D-eigentriplet (see the definition of
the term nonsingular in \Cref{subsec:singularity}).

The rest of this paper is organized as follows.
In \Cref{sec:basic_concepts}, we introduce the concepts
of nonsingularity of a 2D-eigentripet and its 
characterizations, and simple and multiple 2D-eigentriplets and
their properties. 
In \Cref{sec:rqi}, we recap the essential steps of the 2DRQI
presented in \cite{2DEVPII}. 
In \Cref{sec:perturbation}, we recall several known results
on matrix perturbation analysis and {derive a couple of
new results} that will be used for the convergence analysis of the 2DRQI.
\Cref{sec:converge_analysis_I,sec:converge_analysis_II} 
provide convergence analysis of the 2DRQI. 
Conclusion remarks are in \Cref{sec:conclusion}. 


\section{Preliminaries}\label{sec:basic_concepts}

\subsection{Analytical eigencurves and derivatives} 

In Section 4 of Part I \cite{2DEVPI}, we showed that 
equations \eqref{eq:1a} and \eqref{eq:1c} of the 2DEVP~\eqref{2deig}
constitute the parameter eigenvalue problem of $H(\mu) = A-\mu C$.  
For $\mu\in\mathbb{R}$, there exist $n$ real eigenvalues 
$\lambda_i(\mu)$ and corresponding orthonormal eigenvectors of $H(\mu)$. If these eigenvalues $\lambda_i(\mu)$
are sorted such that $\lambda_1(\mu)\ge \cdots \ge \lambda_n(\mu)$, 
then we have $n$ {\em sorted eigencurves} $\lambda_i(\mu)$
for $i = 1, 2, \ldots, n$.
Sorted eigencurves $\lambda_i(\mu)$ might not be differentiable. 
In Part I, we introduced {\em analyticalized eigencurves} 
and {\em analyticalized eigenvector functions} such that they are real 
analytic~\cite[p.\,3]{Krantz2002} on $\mathbb{R}$. The following theorem, 
which is built on Theorem~4.2 of Part I, shows that 
analyticalized eigenvectors enjoy some appealing properties. 

\begin{theorem}\label{Thm:defxmu}
For Hermitian matrices $A$ and $C$,
there exist scalar functions $\widetilde{\lambda}_1(\mu)$, $\cdots$, $\widetilde{\lambda}_n(\mu)$
and matrix-valued functions
$X(\mu)=\begin{bmatrix}x_1(\mu),& \cdots, & x_n(\mu)\end{bmatrix}$ 
of $\mu\in\mathbb{R}$ such that 
\begin{equation}\label{analytic_eigenfuncs}
\begin{aligned}
A-\mu C & =X(\mu)
\Diag\left(\widetilde{\lambda}_1(\mu),\cdots,\widetilde{\lambda}_n(\mu)\right)X^H(\mu), \\
X^H(\mu)X(\mu) &=I.
\end{aligned}
\end{equation}
Furthermore, $\widetilde{\lambda}_i(\mu)$ and $x_i(\mu)$ 
are real analytic on $\mu\in\mathbb{R}$, and $x_i(\mu)$ satisfies
\begin{equation}\label{eq:xHxis0}
x^H_i(\mu) x_i'(\mu)=0,\quad i=1,\cdots,n.
\end{equation}
\end{theorem}
\begin{proof} In Part I, we have shown that 
by \cite[Theorem S6.3]{Gohberg2009Matrix}, 
there exist real analytic scalar functions $\widetilde{\lambda}_1(\mu)$, 
$\cdots$, $\widetilde{\lambda}_n(\mu)$
and real analytic matrix-valued functions
$\widetilde{X}(\mu)=[\widetilde{x}_1(\mu),\,\cdots,\,\widetilde{x}_n(\mu)$ 
of $\mu\in\mathbb{R}$ such that
\begin{equation}
\begin{aligned}
A-\mu C & =\widetilde{X}(\mu)
\Diag\left(\widetilde{\lambda}_1(\mu),\cdots,\widetilde{\lambda}_n(\mu) \right)\widetilde{X}^H(\mu), \\
\widetilde{X}^H(\mu)\widetilde{X}(\mu) &=I.
\end{aligned}
\end{equation}
Therefore, we only need to show that there exists 
a real-valued real analytic function $\theta_i(\mu)$ 
such that $x_i(\mu)=\widetilde{x}_i(\mu){\rm e}^{\ii\theta_i(\mu)}$ 
satisfies \eqref{eq:xHxis0} for $i = 1, \ldots, n$. 
In fact, if we find such 
$\theta_i(\mu)$, then according to the properties~\cite[pp.\,4,19]{Krantz2002} 
of real analytic functions, $x_i(\mu)$ is also real analytic. 
Equation \eqref{analytic_eigenfuncs} holds by defining 
$X(\mu) = [x_1(\mu),\, \cdots,\, x_n(\mu)]$. 
For brevity, the subindex $i$ of $x_i(\mu)$ 
will be dropped in the analysis below.

Since $\widetilde{x}(\mu)$ is real analytic, we can take derivatives 
of $\widetilde{x}^H(\mu)\widetilde{x}(\mu)=1$ and have
\begin{equation}\label{eq:realxpx}
\Real(\widetilde{x}^H(\mu)\widetilde{x}'(\mu))=0.
\end{equation}
On the other hand, equation \eqref{eq:xHxis0} is equivalent to
\begin{equation}\label{eq:xHxpis0equiv}
\widetilde{x}^H(\mu){\rm e}^{-\ii\theta(\mu)}
\left[\widetilde{x}'(\mu){\rm e}^{\ii\theta(\mu)}+
\widetilde{x}(\mu){\rm e}^{\ii\theta(\mu)}\ii\theta'(\mu)\right] 
= \widetilde{x}(\mu)^H\widetilde{x}'(\mu)+\ii\theta'(\mu)=0.
\end{equation}
Equation \eqref{eq:xHxpis0equiv} gives a natural definition for 
$\theta(\mu)$:
\begin{equation}
\theta(\mu) \equiv \int_{0}^{\mu} 
\ii\widetilde{x}^H(s)\widetilde{x}'(s){\rm d}s, \quad \forall \mu\in\mathbb{R}.
\end{equation}
We now prove $\theta(\mu)$ satisfies the desired properties. 
First, by \eqref{eq:realxpx}, $\widetilde{x}^H(\mu)\widetilde{x}'(\mu)$ is 
purely imaginary and thus $\theta(\mu)$ is a real-valued function. 
Furthermore, $\theta'(\mu) = \ii \widetilde{x}^H(\mu)\widetilde{x}'(\mu)$. 
Hence \eqref{eq:xHxpis0equiv} holds and we further have \eqref{eq:xHxis0}. 

To complete the proof, we only need to prove that $\theta(\mu)$ 
is real analytic.
In fact, a function $f$ defined on $\mathbb{R}$ is called 
real analytic \cite[p.\,3]{Krantz2002} if and only if 
for any $\mu_0\in\mathbb{R}$, $f$ has power series 
$$
f(\mu) = \sum_{i=0}^{\infty}f^{(i)}(\mu_0)(\mu-\mu_0)^i
$$ 
with nonzero convergent radius. By separating the real and imaginary 
parts of $f^{(i)}(\mu_0)$, we can see that the real and imaginary parts 
of $f$ are both real analytic. Thus $\widetilde{x}^H(\mu)$ is still real 
analytic, which by properties of real analytic function 
\cite[pp.\,4,11]{Krantz2002} implies $\theta(\mu)$ is real analytic.
\end{proof}

The next lemma presents the derivatives formula for 
$\lambda(\mu)$ and $x(\mu)$.

\begin{lemma}\label{Thm:derivativeFormula}
Let $(\mu_*,\lambda_*,x_*)$ be a 2D-eigentriplet with $\lambda_*$ being a simple eigenvalue of $A-\mu_*C$. Let 
$\lambda(\mu)$ be an analyticalized eigencurve and $x(\mu)$ be the corresponding analyticalized eigenvector function defined in Theorem~\ref{Thm:defxmu} such that $\lambda_*=\lambda(\mu_*)$ and $x_*=x(\mu_*)$. Then we have
\begin{align}
(A-\mu_*C-\lambda_*I)x'(\mu_*)  &= Cx_*,\label{derivativeFormula1}\\
\lambda''(\mu_*)  &= -2x_*^HCx'(\mu_*).\label{derivativeFormula2}
\end{align}
Equation \eqref{derivativeFormula1} can be further written as
\begin{equation}\label{derivativeFormula3}
x'(\mu_*) = (A-\mu_*C-\lambda_*I)^{\dagger}Cx_*,
\end{equation}
where $\cdot^{\dag}$ is the Moore-Penrose generalized inverse.	
\end{lemma}
\begin{proof}
Consider the parameter eigenvalue problem  
\begin{equation} \label{eq:peigeq} 
(A - \mu C) x(\mu) = \lambda(\mu) x(\mu).
\end{equation} 
By taking the derivative with respect to $\mu$, we have
\begin{align} \label{middle1} 
(A-\mu C-\lambda(\mu) I)x'(\mu)  = (C+\lambda'(\mu)I)x(\mu).
\end{align}
Multiplying \eqref{middle1} by $x^H(\mu)$ from left and combining with \eqref{eq:peigeq}, we have   
	\begin{align} \label{eq:secondderi} 
		\lambda'(\mu)  =-x^H(\mu)Cx(\mu).
	\end{align}
	At $\mu=\mu_*$, the derivative becomes 
	\begin{align} \label{eq:lambdader1} 
		\lambda^{\prime}(\mu_*) =-x_*^{H}Cx_*=0.
	\end{align}
	Combined with \eqref{middle1}, we have
	\begin{equation}\label{derivative_1}
		(A - \mu_* C - \lambda_* I)x^{\prime}(\mu_*) = Cx_*.
	\end{equation}
This proves the equation~\eqref{derivativeFormula1}.

To prove the identity \eqref{derivativeFormula2}, we take 
the derivative of \eqref{middle1} with respect to $\mu$ and obtain
\begin{align} 
	(A-\mu C-\lambda(\mu) I)x''(\mu) 
	= 2(C+\lambda'(\mu)I)x'(\mu)+\lambda''(\mu)x(\mu).  \label{middle1b}
\end{align}
Multiplying \eqref{middle1b} by $x^H(\mu)$ from left, 
by \eqref{eq:peigeq}, we have   
\begin{align}
	\lambda''(\mu)  =-2x^H(\mu)(C+\lambda'(\mu)I)x'(\mu). \label{eq:secondderi2}
\end{align}
The equation \eqref{derivativeFormula2} is derived by
taking $\mu=\mu_*$ and the equation \eqref{eq:lambdader1}.

To prove the equation \eqref{derivativeFormula3}, we notice that 
since the multiplicity of $\lambda_*$ is 1, the null subspace 
of $A-\mu_*C-\lambda_*I$ is spanned by $x_*$. Since $(x'(\mu_*))^H x_*=0$ 
by \eqref{eq:xHxis0}, equation~\eqref{derivativeFormula1} implies
the equation \eqref{derivativeFormula3}.

\end{proof}

\subsection{Singularity of 2D-eigentriplets and 
characterizations}\label{subsec:singularity}
In Part II \cite{2DEVPII}, 
we indicated that the 2DEVP \eqref{2deig} can be viewed 
as the problem of finding the root of the following
system of nonlinear equations
\[
F(\mu,\lambda,x) \equiv
\left[ \begin{array}{r}
	(A-\mu C-\lambda I)x  \\
	-x^H C x/2  \\ 
	-x^H x/2 + 1/2 \\ 
\end{array} \right]  = 0.  
\]
The Jacobian of $F$ is defined as 
\begin{equation} \label{eq:jacobi}
J(\mu,\lambda,x)=
\left[
\begin{array}{c|cc}
A-\mu C-\lambda I & -Cx &-x \\
\hline
-x^HC & 0 & 0\\
-x^H & 0 & 0
\end{array}
\right].
\end{equation}
The following definition introduces the notion of singularity of
a 2D-eigentriplet. 

\begin{definition} \label{def:singularity}
A 2D-eigentriplet $(\mu_*, \lambda_*, x_*)$ of $(A,C)$ is called
\emph{nonsingular} if the Jacobian $J(\mu_*, \lambda_*, x_*)$
is nonsingular.  Otherwise, it is called \emph{singular}.
\end{definition}

The following theorem provides characterizations of the singularity.
By these characterizations, we see that the singularity of
a 2D-eigentriplet $(\mu_*,\lambda_*,x_*)$
does not {explicitly} depend on $x_*$. Therefore, we will also call
$(\mu_*,\lambda_*)$ a nonsingular (singular) 2D-eigenvalue.

\begin{theorem}\label{Thm:singularity}
Let $(\mu_*,\lambda_*,x_*)$ be a 2D-eigentriplet of
the 2DEVP \eqref{2deig} and $k$ be 
{the multiplicity of $\lambda_*$ for being an eigenvalue of $A-\mu_*C$}.
\begin{enumerate}[(i)]  
\item \label{Thm:singularity:i}
If $k=1$, then the 2D-eigentriplet $(\mu_*, \lambda_*, x_*)$ is nonsingular
if and only if $\lambda''(\mu_*)\neq0$, where $\lambda(\cdot)$ is an 
analyticalized eigencurve satisfying $\lambda(\mu_*)=\lambda_*$.
		
\item \label{Thm:singularity:ii}
If $k=2$, then the 2D-eigentriplet $(\mu_*, \lambda_*, x_*)$ is 
nonsingular if and only if $\widetilde{V}_*^HC\widetilde{V}_*$ is invertible, 
where $\widetilde{V}_*$ is an orthonormal basis of the eigenspace 
associated with eigenvalue $\lambda_*$ of $A-\mu_* C$.
		
\item \label{Thm:singularity:iii}
If $k\geq3$, then the 2D-eigentriplet $(\mu_*, \lambda_*, x_*)$ is singular.
\end{enumerate}
\end{theorem}
\begin{proof} 
We first note that by $x_*^HCx_*=0$, $x_*$ is orthogonal to $Cx_*$. 
Let $Q$ be an $n\times n$ orthogonal matrix with the first column $x_*$ and 
the second column $q_2$ satisfy $Cx_*=\alpha q_2$. 
Denote $\widetilde{Q}={\rm Diag}(Q,I_2)$. Then
\begin{equation} \label{eq:Jexp}
\widetilde{Q}^HJ_*\widetilde{Q} 
\equiv \widetilde{Q}^HJ(\mu_*, \lambda_*, x_*)\widetilde{Q} 			=\left[\begin{array}{ccc|cc} 0 & 0 & 0  & 0 & -1 \\
0 & a_{22} & a_{23} & -\alpha & 0 \\
0 & a_{23}^H & A_{33} & 0 & 0 \\ \hline
0 &  -\overline{\alpha} &0 & 0 & 0 \\
-1 & 0 & 0 & 0 & 0    \end{array}\right],
\end{equation} 
where $A_{33}$ is $n-2$ by $n-2$.
Let $x(\mu)$ be an analyticalized eigenvector function defined in 
Theorem~\ref{Thm:defxmu}. 
Writing 
\begin{equation} \label{eq:xQ}
x'(\mu_*)=Q[y_1, y_2, y_3^T]^T, 
\end{equation} 
where $y_1$ and $y_2$ are scalars.
Then by  \eqref{derivativeFormula2},
\begin{equation}\label{lampripri}
\lambda''(\mu_*)=-2(Cx_*)^Hx'(\mu_*)=-2\overline{\alpha} y_2.
\end{equation}

Now let us consider the result (i).
By equations \eqref{eq:Jexp} and \eqref{lampripri},
the result (i) is equivalent to prove that 
$J(\mu_*, \lambda_*, x_*)$ is singular if and only if $\lambda''(\mu_*)=0$.

Let us assume $\lambda''(\mu_*)=0$.
By \eqref{lampripri}, it implies $\alpha=0$ or $y_2=0$.
If $\alpha=0$, then by \eqref{eq:Jexp}, $J_*$ is singular since 
it has a zero column. 
If $\alpha\neq0$ and $y_2=0$, we have by \eqref{derivativeFormula1}
\[
Q^H(A-\mu_* C-\lambda_* I)x'(\mu_*)=Q^HCx_*
\]
which implies that 
\begin{align} 
a_{23} y_3 &= \alpha \ (\not=0), \label{eq:derivativeTmp1}\\
A_{33} y_3 &= 0.\label{eq:derivativeTmp2}
\end{align}
Equation~\eqref{eq:derivativeTmp1} says $y_3\not=0$, and 
thus \eqref{eq:derivativeTmp2} admits a nonzero $y_3$. This implies 
$A_{33}$ is singular. By elementary transformation and 
noting $\alpha\neq0$, $\widetilde{Q}^HJ_*\widetilde{Q}$ can be transformed to
\begin{equation}\label{eq:eletransformation}
\left[\begin{array}{ccc|cc} 0 & 0 & 0  & 0 & -1 \\
0 & 0 & 0 & -\alpha & 0 \\
0 & 0 & A_{33} & 0 & 0 \\ \hline
0 &  -\overline{\alpha} &0 & 0 & 0 \\
-1 & 0 & 0 & 0 & 0    \end{array}\right],
\end{equation}
which further implies $Q^HJ_*Q$ is singular since $A_{33}$ is singular. 
Therefore $J_*$ is singular.

On the other hand, assume $J_*$ is singular. If $\alpha=0$, then 
by \eqref{lampripri}, $\lambda^{\prime\prime}(\mu_*) = 0$ and we 
reach the conclusion. If $\alpha\neq0$, by transforming 
$\widetilde{Q}^HJ_*\widetilde{Q}$ to \eqref{eq:eletransformation} 
and noting $J_*$ is singular, we have $A_{33}$ is singular. 
By  \eqref{derivativeFormula1} and \eqref{eq:xQ}, we have
\begin{align}
a_{22}y_2 + a_{23} y_3  &=\alpha \ (\not=0), \label{eq:A331a}\\
a_{23}^Hy_2 + A_{33} y_3&=0.\label{eq:A331b}
\end{align}
Note that since the multiplicity of $\lambda_*$ is 1, 
Matrices $\begin{bmatrix}
			a_{23}\\
			A_{33}
		\end{bmatrix}$ and $\begin{bmatrix}
			a_{22}&a_{23}\\
			a_{23}^H&A_{33}
		\end{bmatrix}$ are of full rank.
Consider the equations 
\begin{equation}\label{eq:A332}
a_{23}z =\alpha, \qquad
A_{33}z =0.
\end{equation}
Since $A_{33}$ is singular, $A_{33}z=0$ admits nonzero solution. 
We further have $a_{23}z\neq0$ since  $\begin{bmatrix}
			a_{23}\\
			A_{33}
		\end{bmatrix}$ is of full rank. 
Thus \eqref{eq:A332} admits a solution, which we denote by $z_*$. Then 
\[
\begin{bmatrix}
a_{22} & a_{23}\\
a_{23}^H & A_{33}
\end{bmatrix}\begin{bmatrix}
0\\
z_*
\end{bmatrix}=\begin{bmatrix}
\alpha\\
0
\end{bmatrix}.
\]
Since $\begin{bmatrix}
a_{22} & a_{23}\\
a_{23}^H & A_{33}
\end{bmatrix}$ is of full rank, 
$[0,z_*^T]^T$ must equal to $[y_2,y_3^T]^T$ according to \eqref{eq:A331a}\eqref{eq:A331b} and thus $y_2=0$. By \eqref{lampripri}, $\lambda^{\prime\prime}(\mu_*)=0$.
This completes the proof of the result (i). 

For the result (ii), we prove that 
$J_*$ is singular if and only if
$\widetilde{V}_*^HC\widetilde{V}_*$ is singular.
Let $Q$ be the orthogonal matrix with its first two columns $\widetilde{V}_*$ and $\widetilde{Q}=\Diag(Q, I_2)$. Then
\[
\widetilde{Q}^HJ_*\widetilde{Q}=\mat{cc|c} O &  
& -\widetilde{V}_*^H[Cx_*,x_*]  \\
& A_{33} & \times \\  \hline
-[Cx_*,x_*]^H\widetilde{V}_* & \times & O\rix, 
\]
where $O$ is a 2-by-2 zero block, $A_{33}$ is nonsingular and 
$\times$ stands for some submatrices.  Obviously, 
\begin{equation}\label{eq:JstarVC}
J_*\text{ is singular} 
\Leftrightarrow
\widetilde{V}_*^H[Cx_*,x_*]\text{ is singular}. 
\end{equation} 
Now for the sake of convenience, we further assume the first column of $\widetilde{V}_*$ is $x_*$, otherwise $\widetilde{V}_*$ will differ from an orthogonal transformation and the conclusion still holds. 
Then from $x_*^HCx_*=0$, we have 
\[
\widetilde{V}_*^H[Cx_*,x_*]=\mat{cc} 0 & 1 \\ \hat{x}^HCx_* & 0 \rix, \quad 
\widetilde{V}_*^HC\widetilde{V}_*
=\mat{cc} 0 & x_*^HC\hat{x} \\ 
\hat{x}^H Cx_* & \hat{x}^HC\hat{x} \rix,
\]
where $\hat{x}$ is the second column of $\widetilde{V}_*$.  Hence
$\widetilde{V}_*^H[Cx_*,x_*]$ is singular 
if and only if 
$\hat{x}^HCx_* =0$, which is equivalent to 
$\widetilde{V}_*^HC\widetilde{V}_*$ is singular.
Together with \eqref{eq:JstarVC}, the result (ii) is proved. 
		
For the result (iii), 
since the multiplicity of $\lambda_*$ is $k$, there exists 
orthogonal $\widetilde{Q}=\Diag(Q, I_2)$ such that
\[
\widetilde{Q}^HJ_*\widetilde{Q}=\mat{c|cc} \begin{array}{cc}O^{k\times k} & \\ & D \end{array} &-Q^HCx_* & -Q^Hx_* \\ \hline
-(Cx_*)^HQ&&\\
-x_*^HQ && \rix.
\]
Since $k\ge 3$, $J_*$ is obviously singular. This completes the proof of
the result (iii).  
\end{proof}

The following theorem provides alternative characterizations 
for the case where the multiplicity of $\lambda_*$ is 2.

\begin{theorem}\label{Thm:singularity_re}
Consider \Cref{Thm:singularity}\eqref{Thm:singularity:ii}, 
denote $\widetilde{\lambda}_i(\cdot)$ and $\widetilde{\lambda}_{i+1}(\cdot)$ 
as the two analyticalized eigencurves of $A-\mu C$ that satisfy 
$\widetilde{\lambda}_i(\mu_*)=\widetilde{\lambda}_{i+1}(\mu_*)=\lambda_*$.  
Then  
\begin{itemize} 
\item[(i)] $(\mu_*,\lambda_*)$ is nonsingular if and only if 
$\widetilde{V}_*^HC\widetilde{V}_*$ is indefinite. 

\item[(ii)]
$(\mu_*,\lambda_*)$ is nonsingular if and only if 
$\tilde{\lambda}_i'(\mu_*)\tilde{\lambda}_{i+1}'(\mu_*)<0$.
\end{itemize} 
\end{theorem}
\begin{proof}
We first prove the equivalence of nonsingularity and indefiniteness 
of $\widetilde{V}_*^HC\widetilde{V}_*$. 
Note that $x_*\in\Span \{\widetilde{V}_*\}$ and $x_*^HCx_*=0$. 
Thus $\widetilde{V}_*^HC\widetilde{V}_*$ is not a definite matrix. 
Therefore, if $\widetilde{V}_*^HC\widetilde{V}_*$ is nonsingular, 
then $\widetilde{V}_*^HC\widetilde{V}_*$ is indefinite. 
On the other hand, if $\widetilde{V}_*^HC\widetilde{V}_*$ is indefinite, 
since the matrix size is 2, $\widetilde{V}_*^HC\widetilde{V}_*$ must 
have one positive and one negative eigenvalue, which implies 
$\widetilde{V}_*^HC\widetilde{V}_*$ is nonsingularity.
	
Now let us consider the statement (i). 
By \Cref{Thm:singularity}\eqref{Thm:singularity:ii}, $(\mu_*,\lambda_*)$ 
is nonsingular if and only if $\widetilde{V}_*^HC\widetilde{V}_*$ is nonsingular. 
By the equivalence of nonsingulairty and indefiniteness of 
$\widetilde{V}_*^HC\widetilde{V}_*$, we reach the statement (i).

For the statement (ii), we note that according to Part I \cite[Thm~4.4]{2DEVPI}, 
the eigenvalues of $-\widetilde{V}_*^HC\widetilde{V}_*$ are 
$\widetilde{\lambda}_i'(\mu_*)$ and $\widetilde{\lambda}_{i+1}'(\mu_*)$. 
Thus $(\mu_*,\lambda_*)$ is nonsingular if and only if  
$\tilde{\lambda}_i'(\mu_*)\tilde{\lambda}_{i+1}'(\mu_*)\neq0$. 
Since $\widetilde{V}_*^HC\widetilde{V}_*$ is not a definite matrix, 
$\tilde{\lambda}_i'(\mu_*)\tilde{\lambda}_{i+1}'(\mu_*)\neq0$ 
is equivalent to $\tilde{\lambda}_i'(\mu_*)\tilde{\lambda}_{i+1}'(\mu_*)<0$.
\end{proof}

We end the discussion of singularity by the following corollary 
from \Cref{Thm:singularity} and equations \eqref{derivativeFormula2} 
and \eqref{eq:JstarVC}.

\begin{corollary}\label{cor:xprimenotzero} ~
\begin{enumerate}[(i)]
\item If $(\mu_*,\lambda_*,x_*)$ is a nonsingular 2D-eigentriplet,
then $Cx_*\neq0$.
\item
If $(\mu_*,\lambda_*)$ is a nonsingular 2D-eigenvalue, 
$\lambda_*$ is an eigenvalue of $A-\mu_*C$ with multiplicity 1, 
and $x(\mu)$ is the corresponding eigenvector function 
defined in \Cref{Thm:derivativeFormula}, then $x'(\mu_*)\neq0$.
\end{enumerate}
\end{corollary}

\subsection{Simple and multiple 2D-eigentriplets} 

\Cref{Thm:singularity,Thm:singularity_re} indicate that 
a 2D-eigenvalue $(\mu_*,\lambda_*)$ is {nonsingular} if and only if 
one of the following two cases happen: 
\begin{enumerate}[I:]
\item\label{case:type1nonsingular} 
$\lambda_*$ is a simple 2D-eigenvalue and the corresponding 
analytic eigencurve $\lambda(\mu)$ satisfies $\lambda''(\mu_*)\neq0$.
	
\item\label{case:type2nonsingular} 
$\lambda_*$ is a multiple 2D-eigenvalue with the multiplicity two 
and the corresponding 
two real analytic eigencurves $\widetilde{\lambda}_1(\mu)$ and 
$\widetilde{\lambda}_2(\mu)$ satisfy 
$\widetilde{\lambda}'_1(\mu_*)\widetilde{\lambda}'_2(\mu_*)<0$. 
\end{enumerate}
Cases~\ref{case:type1nonsingular} and \ref{case:type2nonsingular} 
are the cases of practical interests as we have encountered. 
In the rest of this paper, we will concentrate on the convergence 
analysis of the 2DRQI for these two cases. 
We note that the idea for analyzing Case II can be used for 
treating singular multiple 2D-eigenvalues with corresponding 
analyticalized eigencurves having both negative and positive derivatives. 
For more general cases of singular 2D-eigenvalues, the 2DRQI 
needs to be revised to recover second-order convergence rate.  
It is a subject of future study. 

\subsubsection{Properties of nonsingular simple 2D-eigentriplets}
\label{sec:simple2d} 
Let us consider Case I, namely  
$(\mu_*,\lambda_*, x_*)$ be a nonsingular simple 2D-eigentriplet.
Since it is simple, the set of 2D-eigenvectors is of the form
\begin{equation}\label{eq:XstarStructureI}
\mathcal{X}_* = \{\gamma x_* \mid \gamma \in \mathbb{C}, |\gamma|=1\}. 
\end{equation}
Furthermore, for the Jacobian ${J}(\mu_*,\lambda_*,x_{*})$
and its leading $n$-row matrix $\widehat{J}(\mu_*,\lambda_*,x_{*}) 
= [A-\mu_*C-\lambda_*I, \, -Cx_*\,  -x_*]$, 
$\sigma_{\min}(J(\mu_*,\lambda_*,\gamma x_*))$ and 
$\sigma_{n}(\widehat{J}(\mu_*,\lambda_*,\gamma x_*))$ 
are independent of $\gamma$ for any $\gamma \in \mathbb{C}$ and 
$|\gamma|=1$ due to the facts that
\[
J(\mu_*,\lambda_*,\gamma x_*)
= \Diag(I_n, \bar{\gamma} I_2) J(\mu_*,\lambda_*,x_*) \Diag(I_n, {\gamma} I_2)
\]
and
\[
\widehat{J}(\mu_*,\lambda_*,\gamma x_*)
= \widehat{J}(\mu_*,\lambda_*,x_*) \Diag(I_n, {\gamma} I_2).
\]

On the other hand, by the definition of nonsingularity, 
$\widehat{J}(\mu_*,\lambda_*,x_{*})$ is of full row rank and 
has a nullspace of dimension 2.  
By \Cref{Thm:derivativeFormula},   
a basis matrix of the nullspace of 
$\widehat{J}(\mu_*,\lambda_*,x_{*})$
is given by:
\begin{equation}\label{eq:defVhatstar}
	\widehat{V}_* = \left[\begin{array}{cc}
		x_* & x^{\prime}_*/\sqrt{\|x^{\prime}_*\|^2+1}\\
		0&1/\sqrt{\|x^{\prime}_*\|^2+1}\\
		0&0
	\end{array}\right], 
\end{equation}
where $x^{\prime}_* = x^{\prime}(\mu_*)$, and $x(\mu)$ is 
the analyticalized eigenvector function defined in \Cref{Thm:defxmu} 
corresponding to $(\mu_*,\lambda_*)$ and satisfies $x(\mu_*)=x_*$. 
By \Cref{cor:xprimenotzero}, $x^{\prime}_*\neq0$. 
Let $\widetilde{V}_* = [\, x_*,\, x^{\prime}_*/\|x^{\prime}_*\|\,]$, 
then $\widetilde{V}_*$ is well defined, and has 
orthonormal columns by the orthogonality condition \eqref{eq:xHxis0}. 
Let $V_* =\widetilde{V}_* S$, where $S \in \mathbb{C}^{2 \times 2}$ 
is a unitary matrix such that
\[ 
C_* = V_*^HCV_*
= \Diag(c_{1,*},c_{2,*}) \quad \mbox{with}\quad c_{1,*} \geq c_{2,*}.
\] 
The following lemma presents the properties of the $2\times 2$ 
2DRQ $(A_*, C_*) = (V^H_* AV_*, V^H_*C V_*)$ induced by ${V}_*$. 

\begin{lemma} \label{lem:cstar} 
Let $(\mu_*,\lambda_*,x_*)$ be a nonsingular simple 2D-eigentriplet, 
then 
\begin{enumerate}[(i)]  
\item $C_*$ is indefinite.

\item The (1,2)-element $(A_{*})_{1,2} \neq 0$.
\end{enumerate} 
Therefore, the $2\times 2$ 2DRQ $(A_*,C_*)$ has two simple 2D eigenvalues.
\end{lemma}
\begin{proof} 
For the result (i), it is sufficient to show that $\widetilde{C}_* = \widetilde{V}_*^HC\widetilde{V}_*$ is indefinite. 
	Since $x^{H}_*Cx_* = 0$, and $-2x^{H}_*Cx^{\prime}_* = \lambda^{\prime\prime}(\mu_*)$ from \eqref{derivativeFormula2}, we have
	\[
	\widetilde{C}_* = \widetilde{V}_*^HC\widetilde{V}_*= \left[\begin{array}{cc}
		0&\lambda^{\prime\prime}(\mu_*)/(-2\|x^{\prime}_*\|)\\
		\lambda^{\prime\prime}(\mu_*)/(-2\|x^{\prime}_*\|)&
		(x^{\prime}_*)^H C x^{\prime}_*/\|x^{\prime}_*\|^2
	\end{array}\right].
	\]
Consequently,
\[
\det(\widetilde{C}_*)=-(\lambda^{\prime\prime}(\mu_*)/(2\|x^{\prime}_*\|))^2<0,
	\]
	which implies $\widetilde{C}_*$ is indefinite.

For the result(ii), we use the proof by contradiction. 
For brevity, we use $x,x'$, $x''$,$\mu,\lambda,\lambda',\lambda''$ to 
denote $x_*, x'(\mu_*),x''(\mu_*), \mu_*,\lambda_*,\lambda'(\mu_*),
\lambda''(\mu_*)$. Assume that $(A_*)_{1,2} = (V_*^HAV_*)_{1,2} = 0$. 
By the assumption, $V_*^HAV_*$ and $V_*^HCV_*$ are both diagonal. Therefore, $\widetilde{V}_*^HA\widetilde{V}_*$ and $\widetilde{V}_*^HC\widetilde{V}_*$ are simultaneously diagonalized, and then commute: 
	\[
	(\widetilde{V}_*^HA\widetilde{V}_*)(\widetilde{V}_*^HC\widetilde{V}_*)
	=(\widetilde{V}_*^HC\widetilde{V}_*)(\widetilde{V}_*^HA\widetilde{V}_*),
	\]
	i.e., 
	\[
	\begin{bmatrix}
		\frac{x^HAx'\lambda''}{-2\|x'\|^2}& \frac{\lambda \lambda''}{-2\|x'\|}+\frac{x^HAx'x^{\prime H}Cx'}{\|x'\|^3}\\
		\frac{x^{\prime H}Ax'\lambda''}{-2\|x'\|^3}& \frac{x^{\prime H}Ax\lambda''}{-2\|x'\|^2}+\frac{x^{\prime H}Ax'x^{\prime H}Cx'}{\|x'\|^4}
	\end{bmatrix}=\begin{bmatrix}
		\frac{\lambda''x^{\prime H}Ax}{-2\|x'\|^2}& \frac{x^{\prime H}Ax'\lambda''}{-2\|x'\|^3} \\
		\frac{\lambda \lambda''}{-2\|x'\|}+\frac{x^{\prime H}Axx^{\prime H}Cx'}{\|x'\|^3}& \frac{\lambda''x^HAx'}{-2\|x'\|^2}+\frac{x^{\prime H}Ax'x^{\prime H}Cx'}{\|x'\|^4}
	\end{bmatrix}.
	\]
where we use \eqref{derivativeFormula2}. 
By \Cref{Thm:singularity}(i), $\lambda''\neq 0$. 
Thus the above equation can be simplified as: 
\begin{subequations}\label{a21nonzero}
\begin{empheq}[left={}\empheqlbrace]{alignat=2}
& x^HAx' \ \text{is\ real}, \label{a21nonzero1} \\
& \frac{\lambda''x^{\prime H}Ax'}{-2\|x'\|^3}
=\frac{\lambda \lambda''}{-2\|x'\|}+\frac{x^{\prime H}Cx'x^{\prime H}Ax}{\|x'\|^3}.
\label{a21nonzero2} 
\end{empheq}
\end{subequations}
We now show that equation\eqref{a21nonzero2} implies 
$\lambda''=0$, which contradicts \Cref{Thm:singularity}(i). 
	
	By $(A-\mu C-\lambda(\mu) I)x(\mu)=0$ and the orthogonality~\eqref{eq:xHxis0}, we have
	\begin{equation} \label{eq:xpaxequal} 
		x^{\prime}(\mu)^HAx(\mu)=\mu x^{\prime}(\mu)^HCx(\mu).
	\end{equation}
	By taking the derivative of the identity \eqref{eq:xpaxequal} and setting $\mu=\mu_*$, we have
	\[
	x^{\prime H}Ax'+x^{\prime\prime H}Ax
	=x^{\prime H}Cx+\mu x^{\prime\prime H}Cx+\mu x^{\prime H}Cx'.
	\]
	Thus, 
	\begin{align}
		x^{\prime H}Ax'
		&=-x^{\prime\prime H}(\lambda x+\mu Cx)
		+x^{\prime H}Cx+\mu x^{\prime\prime H}Cx+\mu x^{\prime H}Cx' \nonumber \\
		&=-\lambda x^{\prime\prime H}x+x^{\prime H}Cx+\mu x^{\prime H}Cx'  \nonumber \\
		&=\lambda \|x'\|^2+x^{\prime H}Cx+\mu x^{\prime H}Cx',  \label{eq:xpaxp} 
	\end{align}
	where for the last equality, we use the fact that $x^{\prime\prime H}x=-\|x'\|^2$, which is derived from taking the derivative of the orthogonality condition \eqref{eq:xHxis0} and setting $\mu=\mu_*$.  
	
Thus by \eqref{derivativeFormula2}, \eqref{eq:xpaxequal} and \eqref{eq:xpaxp}, 
equation \eqref{a21nonzero2} is equivalent to
	\[
	\frac{(\lambda \|x'\|^2+x^{\prime H}Cx+\mu x^{\prime H}Cx')\lambda''}{-2\|x'\|^3}
	=\frac{\lambda \lambda''}{-2\|x'\|}+\frac{\mu x^{\prime H}Cx'\lambda''}{-2\|x'\|^3}.
	\] 
	By eliminating the common terms on the both side of the equation, we have $x^{\prime H}Cx=0$. By \eqref{derivativeFormula2}, $\lambda''=0$. This completes 
the proof by contradiction.

By the results (i) and (ii), and \cite[Sec.3]{2DEVPI}, we conclude that the 
2DRQ $(A_*, C_*)$ has two simple 2D-eigenvalues.
\end{proof} 

\subsubsection{Properties of nonsingular multiple 2D-eigentriplets}
\label{sec:propmulti2D} 

Now let us consider Case II, namely 
$(\mu_*, \lambda_*,x_*)$ be a nonsingular multiple 2D-eigentriplet.
Here by \Cref{Thm:singularity}, the nonsingularity implies 
the multiplicity of $\lambda_*$ for being an eigenvalue of $A-\mu_*C$ is 2.
By the definition of nonsingularity, 
the matrix $\widehat{J}_* = [A-\mu_*C-\lambda_*I,\, -Cx_*, \, -x_*]$ 
has full row rank and the dimension of the nullspace of $\widehat{J}_*$ is 2. 
If we denote $\widetilde{V}_*$ as an orthonormal basis of 
the eigenspace associated with $\lambda_*$, a basis matrix of the 
nullspace of $\widehat{J}_*$ is given by 
$[\widetilde{V}^T_*,\, 0, \, 0]^T$. 
Let $V_*=\widetilde{V}_*S$, 
where $S\in\mathbb{C}^{2\times2}$ is an unitary matrix, such that 
\begin{equation}\label{def:VstarII}
C_* \equiv V_*^HCV_* = 
\Diag\left(c_{1,*},c_{2,*}\right)
\quad \mbox{with}\quad 
c_{1,*}\geq c_{2,*}.
\end{equation}
Then 
the set of 2D-eigenvectors is of the form 
\begin{equation}\label{eq:XstarStructureII}
\mathcal{X}_* =
\{\gamma_1t_*v_{1,*}+\gamma_2s_*v_{2,*} \mid
\gamma_1, \gamma_2 \in \mathbb{C}, |\gamma_1|=|\gamma_2|=1 \},
\end{equation}
where $V_*=[v_{1,*},\, v_{2,*}]$,
$t_{*}=\sqrt{{-c_{2,*}}/{(c_{1,*}-c_{2,*})}}$,
$s_{*}=\sqrt{{c_{1,*}}/{(c_{1,*}-c_{2,*})}}$. 

By \Cref{Thm:singularity}, for any $x_*\in\mathcal{X}_*$, 
$J(\mu_*,\lambda_*,x_*)$ and $\widehat{J}(\mu_*,\lambda_*,x_*)$ have full rank. 
Thus 
\[ 
\sigma_{\min}(J(\mu_*,\lambda_*,x_*))>0
\quad \mbox{and} \quad
\sigma_{n}(\widehat{J}(\mu_*,\lambda_*, x_*))>0. 
\] 
Since $\mathcal{X}_*$ is 
compact, $\sigma_{\min}(J(\mu_*,\lambda_*,x_*))$ and 
$\sigma_{n}(\widehat{J}(\mu_*,\lambda_*,x_*))$ are continuous functions 
with respect to $x$ according to 
{Weyl's theorem \cite[p.198]{demmel1997applied}, 
we conclude that 
\[
\sigma_{\min,J_*} =
\inf_{x_*\in\mathcal{X}_*}\sigma_{\min}(J(\mu_*,\lambda_*,x_*))>0
\quad \mbox{and} \quad
\sigma_{n,\widehat{J}_*} =
\inf_{x_*\in\mathcal{X}_*}\sigma_{n}(\widehat{J}(\mu_*,\lambda_*, x_*))>0.
\]

The following lemma presents the properties of the 
$2\times 2$ 2DRQ $(A_*,C_*) = (V_*^HAV_*, V_*^HCV_*)$ 
induced by ${V}_*$. 

\begin{lemma}\label{lem:cstarII}
Let $(\mu_*,\lambda_*, x_*)$ be a nonsingular multiple 
2D-eigenvalue,  then
\begin{enumerate}[(i)]
\item $C_*$ is indefinite;
\item The $(1,2)$-element $(A_*)_{1,2} =0$. 
\end{enumerate} 
Therefore, there are exactly one multiple 2D-eigenvalue of $(A_*, C_*)$.
In addition, we have  
\begin{equation} \label{eq:sigma1C} 
\sigma_{1,C_*} \equiv \inf_{x\in\mathcal{X}_*}\|Cx\|>0 
\quad \mbox{and} \quad
\sigma_{1,VC_*}\equiv \inf_{x\in\mathcal{X}_*}\|V_*^HCx\|
=\sqrt{-c_{1,*}c_{2,*}}>0.
\end{equation} 
\end{lemma}
\begin{proof}
The result (i) is concluded from \Cref{Thm:singularity_re}(i). 
	
For result(ii), note that
\[AV_* = \mu_*CV_*+\lambda_*V_*.\]
Multiplying $V_*^H$ on the left and we obtain
\[V_*^HAV_* = \mu_*\Diag\left(c_{1,*}, c_{2,*}\right)+\lambda_*I.\]
Thus $V_*^HAV_*$ is diagonal and $\left(V_*^HAV_*\right)_{1,2}=0$. 

By the results (i) and (ii), and Section 3 in Part I \cite{2DEVPI}, 
we conclude that the 2DRQ $(A_*, C_*)$ has one multiple 2D-eigenvalue. 

For the two identities in \eqref{eq:sigma1C},
since $\|V_*^HCx\|\leq \|V_*^H\|\|Cx\|=\|Cx\|$, we only need to prove 
the second identity. 
Note that $x\in\mathcal{X}_*$ implies there exists 
$z\in\mathbb{C}^{2}$ satisfying $x = V_*z$, $\|z\|=1$ and 
$z^HC_*z=0$. Denote $z=[z_1, z_2]^T$. Straight calculation shows
\[
\begin{bmatrix}
|z_1|\\
|z_2|
\end{bmatrix} = \frac{1}{\sqrt{c_{1,*}-c_{2,*}}}\begin{bmatrix}
\sqrt{-c_{2,*}}\\
\sqrt{c_{1,*}}
\end{bmatrix},\]
and thus
\[
\|V_*^HCx\| = \|C_*z\| = \sqrt{-c_{1,*}c_{2,*}}>0.
\]
This completes the proof. 
\end{proof} 

Compare the result (ii) in \Cref{lem:cstar,lem:cstarII}, 
we foresee that the convergence behavior of the 2DRQI is 
different for simple and multiple 2D-eigenvalues. This leads to 
different treatments of the convergence analysis of the 2DRQI 
in \Cref{sec:converge_analysis_I,sec:converge_analysis_II}.


\section{Recap of the 2D Rayleigh quotient iteration}\label{sec:rqi}
In this section, we recap the key steps of the 2DRQI presented 
in Section 3 of Part II \cite{2DEVPII}.  
Let $(\mu_k,\lambda_k,x_k)$ be the $k$th approximation of 
a 2D-eigentriplet $(\mu_*,\lambda_*,x_*)$. 
Assume that the Jacobian $J_k \equiv J(\mu_k, \lambda_k, x_k)$ 
is nonsingular, 
see the justification in \Cref{Thm:behavior,Thm:behaviorII}
when $(\mu_k,\lambda_k,x_k)$ is sufficiently close to $(\mu_*, \lambda_*, x_*)$. 
	
Let $\widehat{J}_k = \widehat{J}(\mu_k,\lambda_k,x_k)
= [A-\mu_kC-\lambda_kI,\, -Cx_k, \, -x_k]$ be the first $n$ rows of $J_k$,
and 
$\begin{bmatrix}\widetilde{V}_k\\ R\end{bmatrix}$
be a basis matrix of the nullspace of $\widehat{J}_k$. 
Then the projection matrix $V_k$ of the 2DRQI is defined as 
\begin{equation}\label{eq:vkdef2}
V_k = \mbox{orth}(\widetilde{V}_k)
\end{equation}
and 
\begin{equation}\label{eq:determineVk}
V_k^HCV_k = \Diag(c_{1,k},c_{2,k})
\quad \mbox{with} \quad 
c_{1,k}\geq c_{2,k}. 
\end{equation}
Correspondingly, we have the 2DRQ:
\begin{equation} \label{eq:akck}
(A_k, C_k) \equiv ({V}^H_k A {V}_k, {V}^H_k C {V}_k). 
\end{equation}

In \Cref{thm:Ckindef,lem:CkindefII},  it will be shown that 
when $(\mu_k,\lambda_k,x_k)$ is sufficiently close to 
$(\mu_*, \lambda_*, x_*)$, $C_k$ is indefinite.
Consequently, by Section~3 of Part I \cite{2DEVPI}, 
if $a_{12,k}\neq 0$, where $a_{ij,k}$ is the $(i,j)$ element of $A_k$,  
the $2\times 2$ 2DEVP of 2DRQ  $(A_k, C_k)$: 
\begin{subequations}\label{projectEVP}
\begin{empheq}[left={}\empheqlbrace]{alignat=2}
		(A_k-\nu C_k-\theta I)z &= 0, \label{projectEVP1} \\
		z^HC_kz &= 0, \label{projectEVP2} \\
		z^Hz &= 1, \label{projectEVP3} 
	\end{empheq}
\end{subequations}
has two distinct 2D-eigentriplets
\begin{equation}\label{def:2deigsimple}
(\nu(\alpha_{k,i}),\theta(\alpha_{k,i}),z(\alpha_{k,i})) 
\quad \mbox{for} \quad i = 1,2,  
\end{equation} 
where $\alpha_{k,i} = \pm {|a_{12,k}|}/{a_{12,k}}$, and   
\begin{equation}\label{eq:mulamalpha}
\nu(\alpha) = \frac{z(\alpha)^HC_kA_kz(\alpha)}{\|C_kz(\alpha)\|^2}, \quad
\theta(\alpha) = z(\alpha)^HA_kz(\alpha),\quad
z(\alpha) = \begin{bmatrix}
\sqrt{\frac{-c_{2,k}}{c_{1,k}-c_{2,k}}}\\
\alpha\sqrt{\frac{c_{1,k}}{c_{1,k}-c_{2,k}}}
\end{bmatrix}.
\end{equation}
Otherwise, if $a_{12,k} = 0$, the 2D-eigentriplet of 
the $2\times 2$ 2DEVP \eqref{projectEVP} is 
\begin{equation} \label{eq:2deigsmultiple} 
(\nu_1,\theta_1,z(\alpha))\equiv 
\left(\frac{a_{11,k}-a_{22,k}}{c_{1,k}-c_{2,k}}, 
\frac{a_{22,k}c_{1,k}-a_{11,k}c_{2,k}}{c_{1,k}-c_{2,k}}, z(\alpha)\right),
\end{equation} 
where $\alpha \in \mathbb{C}$ and $|\alpha|=1$.

From the 2D-eigentriplets \eqref{def:2deigsimple} and \eqref{eq:2deigsmultiple} 
of the 2DRQ $(A_k, C_k)$, 
when $a_{12,k}\neq 0$, by \eqref{def:2deigsimple}
the following 2D Ritz triplet defines  
the $k+1$st approximate 2D-eigentriplet of $(A,C)$:  
\begin{equation}  \label{eq:update1a}
\mu_{k+1}=\nu(\alpha_{k,j}), \quad
\lambda_{k+1}=\theta(\alpha_{k,j}) \quad \mbox{and} \quad
x_{k+1} = V_kz(\alpha_{k,j}),
\end{equation}
where the index $j$ is the one such that 
$|\mu_k-\nu(\alpha_{k,j})|+|\lambda_k-\theta(\alpha_{k,j})|$ is smaller 
for $j = 1, 2$. 
Otherwise, when $a_{12,k} = 0$, by \eqref{eq:2deigsmultiple},  
the $k+1$st approximate 2D-eigentriplet of $(A,C)$ is given by
\begin{equation}  \label{eq:update1b}
\mu_{k+1}= \nu_1, \quad
\lambda_{k+1}= \theta_1 
\quad \mbox{and} \quad
x_{k+1} = V_kz(1), 
\end{equation}
where for brevity, we choose $\alpha=1$.



\section{Results from matrix perturbation analysis} 
\label{sec:perturbation} 

In this section, we recall several known results 
on matrix perturbation analysis and {derive a couple of 
new results} that will be used for the convergence analysis of the 2DRQI.  

The canonical angles provide a useful tool to measure the distance 
between two subspaces, see e.g.,~\cite[Sec 4.2.1]{Stewart1990Sun}.
Let $X, Y\in \mathbb{C}^{n\times k}$ and have orthonormal columns, $k \leq n$.  
Then the $k$ canonical angles $\theta_j(\mathcal X,\mathcal Y)$ 
between the range spaces of 
$\mathcal X =\mathcal{R}(X)$ and
$\mathcal Y = \mathcal{R}(Y)$ are defined by
\begin{equation}\label{eq:indv-angles-XY}
0 \le \theta_j(\mathcal X,\mathcal Y)
:=\arccos\sigma_j\le \frac {\pi}2\quad\mbox{for $1\le j\le k$},
\end{equation}
where $\sigma_1\ge\cdots\ge\sigma_k$ are singular values of
the matrix $Y^{H}X$, and are in ascending order
$$
\theta_1(\mathcal{X},\mathcal{Y} ) \le\cdots\le
\theta_k(\mathcal{X},\mathcal{Y}).
$$
Let 
\begin{equation}\label{eq:mat-angles-XY}
\Theta(\mathcal X,\mathcal Y)= \mbox{Diag}(\theta_1(\mathcal X,\mathcal Y),
\ldots,\theta_k(\mathcal X,\mathcal Y)).
\end{equation}
It is well-known that for any unitarily invariant norm $\|\cdot\|_{\rm UI}$,
it holds that both $\|\Theta({X},{Y})\|_{\rm UI}$ 
and $\|\sin\Theta({ X},{ Y})\|_{\rm UI}$ are
unitarily invariant metrics on the Grassmann manifold 
$\mathbf{Gr}(k,\mathbb{C}^n)$ (see e.g.,~\cite[Thm.4.10,p.93]{Stewart1990Sun}).
Note that since the canonical angles are independent of 
the basis matrices $X$ and $Y$, for convenience, we use the notation
$\Theta({X},{Y})$ interchangeably with $\Theta(\mathcal X,\mathcal Y)$.

The following result from \cite[p.\,92, Thm~4.5]{Stewart1990Sun} 
expresses the metric $\|\sin\Theta(U,V)\|$
in terms of the distance between a vector $x$ and a closed set $Y$: 
\[
\dist(x,Y)=\min\{\|x-y\| \mid y\in Y\}. 
\]
\begin{theorem}[\cite{Stewart1990Sun}] \label{Thm:sinthetais}
Assume $U,V\in\mathbb{C}^{n\times \ell}$ are of full column rank, 
then 
\[
\|\sin\Theta(U,V)\|=\max\left\{\max\limits_{u\in\Span\{U\}, \|u\|=1}
\dist(u,\Span\{V\}), \max\limits_{v\in\Span\{V\}, \|v\|=1}
\dist(v,\Span\{U\})\right\}.
\]
\end{theorem}

By \Cref{Thm:sinthetais}, we have the following result.  

\begin{lemma}\label{lemma:sinthetatwovec}
Assume $u,v\in\mathbb{C}^n, \|u\| = \|v\| = 1$, then 
\[ \|\sin\Theta(u,v)\| =  \sqrt{1-|u^Hv|^2}.\]
\end{lemma}
\begin{proof}
We only need to note that
\begin{align*}
\max\limits_{z\in\Span \{u\}, \|z\|=1}\dist(z,\Span\{v\})
& =\max\limits_{ \|\gamma\|=1}\dist(\gamma u,\Span\{v\})\\
 &=\dist(u,\Span\{v\}) =\| (I- vv^H)u\| =\sqrt{1-|u^Hv|^2}.
\end{align*} 
\end{proof}

The next theorem shows that for two matrices $U,V$ such that 
$\|U-V\|$ is small, 
when one of them has orthonormal columns, 
$\|\sin\Theta(U,V)\|$ will also be small. 

\begin{theorem}\label{Thm:lu}
Let $U,V \in\mathbb{C}^{n\times l}$. Assume $U$ has orthonormal 
columns, and $\|U-V\|\leq\frac{1}{2}$, then $\|\sin\Theta(U,V)\|\leq 2\|U-V\|$. 
\end{theorem}
\begin{proof}
Since $U$ has orthonormal columns, we have 
\begin{equation} \label{eq:bd1} 
\begin{aligned}
\max\limits_{u\in\Span\{U\}, \|u\|=1}\dist(u,\Span\{V\})
&=\max\limits_{\|z\|=1}\dist(Uz,\Span\{V\})\\
&\leq \max\limits_{\|z\|=1}\|(U-V)z\| = \epsilon,
\end{aligned}
\end{equation} 
where $\epsilon=\|U-V\|$.  
On the other hand, 
\begin{equation}  \label{eq:bd2} 
\begin{aligned}
\max\limits_{v\in\Span \{V\}, \|v\|=1}\dist(v,\Span\{U\})
&=\max\limits_{\|Vz\|=1}\dist(Vz,\Span U)\\
&\leq \max\limits_{\|Vz\|=1}\|(U-V)z\|
\leq \max\limits_{\|Vz\|=1}\epsilon\|z\|.
\end{aligned}
\end{equation} 
Denote $E=V-U$, 
then we have
\begin{align*}
1 & = \|Uz+E z\|^2 = \|z\|^2+\|E z\|^2+2\Real(z^HU^HE z) \\
  & \geq \|z\|^2+\|E z\|^2-2\|z\|\|E z\|
=(\|z\|-\|Ez\|)^2 \geq (1-\epsilon)^2\|z\|^2.
\end{align*}
Thus $\|z\|\leq \frac{1}{1-\epsilon}$, which implies 
\begin{equation} \label{eq:bd3}
\max\limits_{\|Vz\|=1}\epsilon\|z\|\leq 
\frac{\epsilon}{1-\epsilon}\leq 2\epsilon. 
\end{equation} 
The proof is completed by 
combining \eqref{eq:bd1}, \eqref{eq:bd2}, \eqref{eq:bd3} and
\Cref{Thm:sinthetais}.
\end{proof}

The following result~\cite[lemma~4.1]{Zhang2015} relates 
$\|\sin\Theta\|$ metric and usual $\|\cdot\|$.

\begin{lemma}\label{lemma:Zhang}
Let $U,V$ be $n\times \ell$ matrices with orthonormal columns, 
then there exists a unitary matrix $Z\in\mathbb{C}^{\ell\times \ell}$, 
such that
\[
\|\sin\Theta(U,V)\|\leq \|U-VZ\|\leq\sqrt{2}\|\sin\Theta(U,V)\|.
\]
\end{lemma}

The following is the well-known $\sin\Theta$ theorem due to
Davis and Kahan~\cite{Davis1970}.

\begin{theorem}[\cite{Davis1970}]\label{Thm:sinTheta}
        Let $A$ and $A+H$ satisfy
        \[ \begin{bmatrix}
                X_1^H\\
                X_2^H
        \end{bmatrix}A\begin{bmatrix}
                X_1&X_2
        \end{bmatrix}=\Diag(A_1,A_2),\quad \begin{bmatrix}
        Y_1^H\\
        Y_2^H
\end{bmatrix}(A+H)\begin{bmatrix}
Y_1& Y_2
\end{bmatrix}=\Diag(L_1,L_2),\]
where $[X_1,X_2]$ and $[Y_1,Y_2]$ are unitary with
$X_1,Y_1\in\mathbb{C}^{n\times k}$. Let
\[
R = (A+H)X_1-X_1A_1,
\]
{where $A_1\in\mathbb{C}^{k\times k}$}.
If there exists $\delta>0$ and an interval $[\alpha,\beta]$, such that
\[
\Lambda(A_1)\subseteq [\alpha,\beta],
\qquad
\Lambda(L_2)\subseteq \mathbb{R}\setminus(\alpha-\delta,\beta+\delta),
\]
where $\Lambda(X)$ denotes the set of eigenvalues of the matrix $X$, then
\[
\|\sin\Theta\left(X_1,Y_1\right)\|\leq\frac{\|R\|}{\delta}.
\]
\end{theorem}

Next we present a couple of results derived from the above
classical results.
We begin from the following perturbation theorem 
for the nullspace {of a matrix} based on \Cref{Thm:sinTheta}).

\begin{theorem}\label{Thm:null}
Let $J, \widetilde{J} \in\mathbb{C}^{k\times n}$ be of full row rank, $k < n$, 
and $X_1$ and $\widetilde{X}_1$ be the orthonormal bases 
of the nullspaces of $J$ and $\widetilde{J}$, respectively, 
Assume 
$\epsilon = \|J-\widetilde{J}\| \leq\frac{1}{2}\sigma_{\min}(J)$,
where $\sigma_{\min}(J)$ is the smallest singular value of $J$. 
Then 
\begin{equation} \label{eq:JJbd} 
\|\sin\Theta(X_1,\widetilde{X}_1)\|\leq 
\frac{8\|J\|} {\sigma^2_{\min}(J)}\,\epsilon.
\end{equation} 
\end{theorem}
\begin{proof}
Note that $Jx=0$ if and only if $J^HJx=0$. Thus $X_1$ and 
$\widetilde{X}_1$ are also the orthonormal bases 
of eigen-subspace corresponding to the eigenvalue 0 
of $J^HJ$ and $\widetilde{J}^H\widetilde{J}$, respectively. 
	
By Weyl's theorem~\cite[p.198]{demmel1997applied},
$\sigma_{\min}(\widetilde{J})\geq \frac{1}{2}\sigma_{\min}(J)$. 
To apply \Cref{Thm:sinTheta}, let $\widetilde{X}_2$ be 
the $n\times (n-k)$ matrix such that 
$[\widetilde{X}_1,\widetilde{X}_2]$ is a unitary matrix, and 
let $A_1$ be a $k\times k$ zero matrix, then we have
\[
\begin{aligned}
&\begin{bmatrix}X_1^H\\ X_2^H\end{bmatrix}J^HJ\begin{bmatrix}X_1&X_2\end{bmatrix} = \Diag(A_1,A_2),\quad \begin{bmatrix}\widetilde{X}_1^H\\\widetilde{X}_2^H\end{bmatrix}\widetilde{J}^H\widetilde{J}\begin{bmatrix}\widetilde{X}_1&\widetilde{X}_2\end{bmatrix} = \Diag(0,L_2),\\
&\Lambda(A_1) \subseteq[0,0],\quad 
\Lambda(L_2)\subseteq \mathbb{R}\setminus 
\left(-\frac{1}{4}\sigma^2_{\min}(J),\frac{1}{4}\sigma^2_{\min}(J) \right).
\end{aligned}
\]
Let $R\equiv \widetilde{J}^H\widetilde{J}X_1-X_1A_1$.  
Then by a straightforward calculation we have
\[  \|R\|=\|\widetilde{J}^H(\widetilde{J}-J)X_1\|\leq 
\epsilon(\epsilon+\|J\|)\leq 2\|J\|\epsilon, \]
where the last inequality results from the fact that 
$\epsilon\leq \frac{1}{2}\sigma_{\min}(J)\leq\|J\|$. 
Then the bound \eqref{eq:JJbd} is directly derived from \Cref{Thm:sinTheta}.
\end{proof}

The following theorem shows that for matrices with two orthonormal columns, 
if they simultaneously diagonalize a Hermitian matrix $C$ and 
their $\|\sin\Theta\|$ metric is small, then their 2-norm difference 
could also be small under a column scaling. 
It will be used in \Cref{Thm:behavior2,lemmaRQItypeII} for proving 
the approximation properties of 2DRQs.

\begin{theorem}\label{Thm:doubleV}
Let $C \in \mathbb{C}^{n\times n}$ be Hermitian. 
Assume $U=[u_1, u_2]$ and $V = [v_1, v_2]$ have orthonormal columns 
such that $U^HCU=\Diag(c_{1,u},c_{2,u})$ with {$c_{1,u} > c_{2,u}$}, 
and $V^HCV=\Diag(c_{1,v},c_{2,v})$ with $c_{1,v}\geq c_{2,v}$ 
%
%
Then there exist positive constants $t_0$ and $\kappa_0$ depending on 
$(C,c_{1,u}-c_{2,u})$, such that when $\|\sin\Theta(U,V)\|\leq t_0$, 
\[
\left\|U-V\begin{bmatrix} \gamma_1&\\ &\gamma_2 \end{bmatrix}\right\|
\leq \kappa_0\|\sin\Theta(U,V)\|.
\]
where $\gamma_i=\sign(v_i^Hu_i)$ for $i=1,2$. 
\end{theorem}
\begin{proof}
By \Cref{lemma:Zhang}, there exists unitary matrices $Z$ 
such that 
\[
\|U-VZ\|\leq\sqrt{2}t,
\]
where $t=\|\sin\Theta(U,V)\|$. 
Denote $\Delta U = U-VZ$, then we have
\begin{equation}\label{eq:lemdoubleVeqs0}
\|U^HCU-Z^HV^HCVZ\|=\|\Delta U^HCU+Z^HV^HC\Delta U\|\leq 2\sqrt{2}\|C\|t.
\end{equation}
Note that $c_{1,v}$ and $c_{2,v}$ are eigenvalues of $Z^HV^HCVZ$. 
By Weyl's theorem~\cite[p.\,198]{demmel1997applied},
for $i=1,2$, we have
\[
|c_{i,u}-c_{i,v}|\leq 2\sqrt{2}\|C\|t.
\]
Let $W=[w_1,w_2]\equiv Z^H$, consider the eigenvalue decomposition of 
$U^HCU$ and $Z^HV^HCVZ$: 
\[
\begin{aligned}
U^HCU = \Diag(c_{1,u}, c_{2,u}),\quad 
Z^HV^HCVZ = W \Diag(c_{1,v}, c_{2,v})W^H.
\end{aligned}
\]
Let $R=(Z^HV^HCVZ-U^HCU)e_1$. Utilizing \eqref{eq:lemdoubleVeqs0}, we have
\[
\|R\|\leq 2\sqrt{2}\|C\|t. 
\]
Note that when $t\leq \frac{c_{1,u}-c_{2,u}}{4\sqrt{2}\|C\|}$, 
$c_{2,v}\leq c_{2,u}+\frac{c_{1,u}-c_{2,u}}{2}=\frac{c_{1,u}+c_{2,u}}{2}$. 
Denote  $\delta = \frac{c_{1,u}-c_{2,u}}{2}$, we have		
\[
c_{1,u}\subseteq [c_{1,u},c_{1,u}] 
\quad \mbox{and} \quad
c_{2,v}\subseteq \mathbb{R}\setminus (c_{1,u}-\delta,c_{1,u}+\delta). 
\] 
Thus by \Cref{Thm:sinTheta}, 
\begin{equation}\label{eq:lemdoubleVeqs1}
\|\sin\Theta(e_1,w_1)\|\leq \frac{2\sqrt{2}\|C\|t}{\delta} 
\quad \mbox{and} \quad 
\|\sin\Theta(e_2,w_2)\|\leq \frac{2\sqrt{2}\|C\|t}{\delta}.
\end{equation}
On the other hand, according to \Cref{lemma:sinthetatwovec}, 
for $i = 1, 2$, 
\begin{equation}\label{eq:lemdoublesinglevec} 
\|\sin\Theta(e_i,w_i)\| = \sqrt{1-|e_i^Hw_i|^2}.
\end{equation}
By \eqref{eq:lemdoubleVeqs1} and \eqref{eq:lemdoublesinglevec},
\begin{equation}\label{eq:lemdoubleVeqs3}
 1-|e_i^Hw_i|\leq 1-|e_i^Hw_i|^2 \leq \frac{8\|C\|^2}{\delta^2}t^2.
\end{equation}
Now define $\widetilde{\gamma}_i=\sign(w_i^He_i)$ and we have
\begin{equation}
\begin{aligned}
\min\limits_{|\widehat{\gamma}_1|=|\widehat{\gamma}_2|=1}\|U-V\Diag(\widehat{\gamma}_1,\widehat{\gamma}_2)\|_F&\leq\|U-V\Diag(\widetilde{\gamma}_1,\widetilde{\gamma}_2)\|_F\\
&\leq \|U-VZ\|_F+\|VZ-VZW\Diag(\widetilde{\gamma}_1,\widetilde{\gamma}_2)\|_F\\
&\leq 2t + \|I-W\Diag(\widetilde{\gamma}_1,\widetilde{\gamma}_2)\|_F\\
&= 2t +  \sqrt{4-2|e_1^Hw_1|-2|e_2^Hw_2|}\\
&\leq 2t +  \sqrt{\frac{32\|C\|^2}{\delta^2}t^2}
=2t +  \frac{4\sqrt{2}\|C\|}{\delta}t.
\end{aligned}
\end{equation}
where in the third inequality we use the fact that 
$\|X\|_F=\sqrt{\|X(:,1)\|^2+\|X(:,2)\|^2}\leq\sqrt{2}\|X\|$ 
for a $n\times2$ matrix. 

On the other hand, for $|\widehat{\gamma}_1|=|\widehat{\gamma}_2|=1$, $\|U-V\Diag(\widehat{\gamma}_1,\widehat{\gamma}_2)\|_F^2 = 4-2\Real(u_1^Hv_1\widehat{\gamma}_1)-2\Real(u_2^Hv_2\widehat{\gamma}_2)$ reach minimum when $\widehat{\gamma}_i=\sign(v_i^Hu_i)$. Since $\gamma_i=\sign(v_i^Hu_i)$, we have
\[
\begin{aligned}
\|U-V\Diag(\gamma_1,\gamma_2)\|\leq \|U-V\Diag(\gamma_1,\gamma_2)\|_F&=\min\limits_{|\widehat{\gamma}_1|=|\widehat{\gamma}_2|=1}\|U-V\Diag(\widehat{\gamma}_1,\widehat{\gamma}_2)\|_F\\ &\leq \left(2 +  \frac{4\sqrt{2}\|C\|}{\delta}\right)t.  
\end{aligned}
\]
Let $t_0 = \frac{c_{1,u}-c_{2,u}}{4\sqrt{2}\|C\|}$ and 
$\kappa_0 = 2+\frac{4\sqrt{2}\|C\|}{\delta}$, then we reach the conclusion. 
\end{proof}

To end this section, we present a simple estimate on a second-order approximation
of $\lambda_*$ from an approximate 2D-eigenvector. 
It will be used in \Cref{lemma:xkp,Thm:verylong} 
for the error bounds of 2D-Ritz values. 

\begin{theorem}\label{Thm:lsq}
Let $(\mu_*,\lambda_*,x_*)$ be a 2D-eigentriplet of the 2DEVP \eqref{2deig},
Assume $\widetilde{x}$ is 
an approximate 2D-eigenvector satisfying $\widetilde{x}^HC\widetilde{x}=0$ 
and $\widetilde{x}^H\widetilde{x}=1$. 
Let $\epsilon=\|x_*-\widetilde{x}\|$, then
\begin{equation} \label{eq:2ndbd} 
|\widetilde{x}^HA\widetilde{x}-\lambda_*|\leq 
\|A-\mu_*C-\lambda_*I\|\,\epsilon^2.
\end{equation} 
\end{theorem}
\begin{proof}
Denote $\Delta x=x_*-\widetilde{x}$, by striaghtforward calculation, 
we have
\begin{equation}\label{eq:lsqeqsN0}
\begin{aligned}		
|\widetilde{x}^HA\widetilde{x}-\lambda_*|&=|(x_*-\Delta x)^HA(x_*-\Delta x)-\lambda_*|\\
&=|-2\Real(\Delta x^HAx_*)+\Delta x^HA\Delta x|\\
&=|-2\Real(\Delta x^H(\mu_*Cx_*+\lambda_*x_*))+\Delta x^HA\Delta x|,
\end{aligned}
\end{equation}
where in the second inequality we use the fact 
$\lambda_*=x_*^HAx_*$.
By $\widetilde{x}^HC\widetilde{x}=0$ and $\widetilde{x}^H\widetilde{x}=1$, we have
\[
\begin{aligned}
0 &=(x_*-\Delta x)^HC(x_*-\Delta x) = -2\Real(\Delta x^HCx_*)+\Delta x^HC\Delta x,\\
1 &=(x_*-\Delta x)^H(x_*-\Delta x) = 1-2\Real(\Delta x^Hx_*)+\Delta x^H\Delta x,\\
\end{aligned}
\]
which implies 
\begin{equation} \label{eq:2Rebd} 
2\Real(\Delta x^HCx_*) = \Delta x^HC\Delta x
\quad \mbox{and} \
2\Real(\Delta x^Hx_*) = \Delta x^H\Delta x.
\end{equation} 
The desired bound \eqref{eq:2ndbd} is then derived by 
combining \eqref{eq:lsqeqsN0} and \eqref{eq:2Rebd}. 
\end{proof}

\section{Convergence analysis for simple-2D eigenvalues}
\label{sec:converge_analysis_I}

In this section, we prove that the 2DRQI is locally quadratically convergent for computing a nonsingular simple 2D-eigentriplet.  

%

\subsection{Properties of Jacobian $J_k$} 
Let $(\mu_k, \lambda_k, x_k)$ be the $k$-th iterate to a nonsingular simple 2D-eigentriplet $(\mu_*,\lambda_*,x_*)$, where $x_*$ is the vector in $\mathcal{X}_*$ closest to $x_k$ and $\mathcal{X}_*$ is the set of 2D eigenvectors defined in \eqref{eq:XstarStructureI}. The following lemma shows that when $(\mu_k,\lambda_k,x_k)$ is sufficiently close to 
$(\mu_*,\lambda_*,\mathcal{X}_*)$, Jacobian $J_k = J(\mu_k, \lambda_k, x_k)$ is nonsingular.

\begin{lemma}\label{Thm:behavior}
Let $\epsilon = \max\{|\mu_k-\mu_*|, |\lambda_k-\lambda_*|, 
\dist(x_k,\mathcal{X}_*)\}$. Then there exists $\epsilon_1>0$ 
depending on $(A,C,\mu_*,\lambda_*)$, such that if 
$\epsilon\leq\epsilon_1$, the following statements hold:  
\begin{enumerate}[(i)] 
\item $J_k$ is nonsingular.

\item
$\sigma_{n}(\widehat{J}_k) \geq\frac{1}{2}\,\sigma_{n}(\widehat{J}_*)$.
\end{enumerate} 
\end{lemma}  
\begin{proof} 
Note that
\begin{equation} \label{eq:Thmbehavioreqs1}
\|J_k-J_*\|=\left\|\begin{bmatrix}
(\mu_*-\mu_k)C+(\lambda_*-\lambda_k)I&C(x_*-x_k)& x_*-x_k\\
(x_*-x_k)^HC&0&0\\
(x_*-x_k)^H&0&0
\end{bmatrix}\right\|\leq 3(\|C\|+1)\epsilon,
\end{equation}
and
\begin{equation}\label{eq:Thmbehavioreqs2}
\|\widehat{J}_k-\widehat{J}_*\|=\left\|\begin{bmatrix}
(\mu_*-\mu_k)C+(\lambda_*-\lambda_k)I&C(x_*-x_k)& x_*-x_k\\
\end{bmatrix}\right\|\leq \sqrt{3}(\|C\|+1)\epsilon.
\end{equation}
Here the last inequalities in \eqref{eq:Thmbehavioreqs1}
and \eqref{eq:Thmbehavioreqs2} are from the following 
matrix norm inequality~\cite[Lemma 3.5]{2007Higham}: 
\begin{equation}
\left\|\begin{bmatrix}
A_{11}&\cdots&A_{1n}\\
\vdots&\ddots&\vdots\\
A_{m1}&\cdots&A_{mn}
\end{bmatrix}\right\|\leq\sqrt{mn}\max\limits_{i,j}\|A_{ij}\|. 
\end{equation}
Let
\begin{equation}\label{eq:defepsilon1}
\epsilon_1=
\frac{\min\{\sigma_{\min}(J_*), \sigma_{n}(\widehat{J}_*)\}}{6(\|C\|+1)}>0,
\end{equation}
Then when $\epsilon\leq\epsilon_1$, by 
Weyl's theorem\cite[p.\,198]{demmel1997applied}, we have 
\begin{equation} \label{eq:jksigma}  
|\sigma_{\min}(J_k)-\sigma_{\min}(J_*)|\leq\frac{\sigma_{\min}(J_*)}{2}
\end{equation} 
and
\begin{equation} \label{eq:jkhatsigma} 
|\sigma_{n}(\widehat{J}_k)-\sigma_{n}(\widehat{J}_*)|
\leq\frac{\sigma_{n}(\widehat{J}_*)}{2}.
\end{equation} 
By \eqref{eq:jksigma} and \eqref{eq:jkhatsigma}, we have
the results (i) and (ii). 
\end{proof}

\subsection{Approximation of $V_k$ to $V_*$}
Next we show that the approximation of $V_k$ to $V_*$ 
when $(\mu_k, \lambda_k, x_k)$ is sufficiently 
close to $(\mu_*, \lambda_*, \mathcal{X}_*)$. 
Recall that $V_k$ is defined in \eqref{eq:vkdef2} 
and \eqref{eq:determineVk} and $V_*$ is defined in \Cref{lem:cstar}.

\begin{lemma}\label{lemmaRQItypeI}
Let $\epsilon = \max\{|\mu_k-\mu_*|, |\lambda_k-\lambda_*|, 
\dist(x_k,\mathcal{X}_*)\}$. There exists positive scalars 
$\epsilon_t, \kappa_t$ depending on $(A,C,\mu_*,\lambda_*)$, 
such that if $\epsilon\leq\epsilon_t$, we have 
\begin{equation}\label{eq:VkV*}
\|\sin\Theta(V_k,V_*)\|\leq \kappa_t\epsilon.
\end{equation} 
\end{lemma}
\begin{proof}
Consider the nullspace of $\widehat{J}_k$. 
According to \Cref{Thm:behavior}, there exists positive constants 
$\epsilon_1$ depending on $(A,C,\mu_*,\lambda_*)$, 
such that if $\epsilon\leq\epsilon_1$, we have
\[
\sigma_{n}(\widehat{J}_k)\geq\frac{1}{2}\sigma_{n}(\widehat{J}_*)>0.
\]
Thus the nullspace of $\widehat{J}_k$ has dimension 2. 
Denote $\widehat{V}_k$ as an orthonormal basis of 
the nullspace of $\widehat{J}_k$.
	
On the other hand, in \Cref{sec:simple2d}, we know that  
\[
	\widehat{V}_*\equiv \left[\begin{array}{cc}
		x_* & x^{\prime}_*/\sqrt{\|x^{\prime}_*\|^2+1}\\
		0&1/\sqrt{\|x^{\prime}_*\|^2+1}\\
		0&0
	\end{array}\right], 
\]
is an orthonormal basis of the nullspace of $\widehat{J}(\mu_*,\lambda_*,x_*)$. 
Note that by \eqref{eq:Thmbehavioreqs2}, 
\[
\|\widehat{J}_*-\widehat{J}_k\|\leq\sqrt{3}(\|C\|+1)\epsilon.
\]
Since $\epsilon\leq\epsilon_1$, by the definition \eqref{eq:defepsilon1}
of $\epsilon_1$, we have 
\[
\sqrt{3}(\|C\|+1)\epsilon\leq\frac{\sigma_{n}(\widehat{J}_*)}{2}.  
\] 
By \Cref{Thm:null}, we obtain
\begin{equation} \label{eq:sinbd} 
\begin{aligned}
\|\sin\Theta(\widehat{V}_*,\widehat{V}_k)\|&\leq  \frac{8\sqrt{3}(\|C\|+1)\|\widehat{J}_*\|\epsilon}{\sigma_{n}(\widehat{J}_*)^2}\\
&\leq\frac{8\sqrt{3}(\|C\|+1)(\|A-\mu_*C-\lambda_*I\|+\|C\|+1)}{\sigma_{n}(\widehat{J}_*)^2}\epsilon
\equiv\widetilde{\alpha}_1\epsilon.	
\end{aligned}
\end{equation} 
By \Cref{lemma:Zhang}, {the inequality~\eqref{eq:sinbd}}  
implies there exists a unitary matrix $\widehat{Z}$ such that
\[
\|\widehat{V}_*-\widehat{V}_k\widehat{Z}\|\leq
\sqrt{2}\|\sin\Theta(\widehat{V}_k, \widehat{V}_*)\|
=\sqrt{2}\widetilde{\alpha}_1\epsilon.
\]
{Therefore, for the first $n$ rows of $\widehat{V}_*-\widehat{V}_k\widehat{Z}$,}
we have
\[
\left\|V_*\Diag(1,\frac{\|x_*^{\prime}\|}{\sqrt{\|x_*^{\prime}\|^2+1}})
-\widehat{V}_k(1:n,:)\widehat{Z}\right\|
\leq\sqrt{2}\widetilde{\alpha}_1\epsilon, 
\]
or equivalently, 
\[
\left\|V_*-\widehat{V}_k(1:n,:)\widehat{Z}\Diag\left(1,\frac{\sqrt{\|x_*^{\prime}\|^2+1}}{\|x_*^{\prime}\|}\right)\right\|\leq\sqrt{2}\frac{\sqrt{\|x_*^{\prime}\|^2+1}}{\|x_*^{\prime}\|}\widetilde{\alpha}_1\epsilon.\]
By \Cref{Thm:lu}, when $\epsilon\leq\frac{\|x_*^{\prime}\|}{2\sqrt{2}\widetilde{\alpha}_1\sqrt{\|x_*^{\prime}\|^2+1}}$, we have
	\begin{equation}\label{eq:typeIIImportantLemma}
		\|\sin\Theta(V_*,V_k\|\leq2\sqrt{2}\frac{\sqrt{\|x_*^{\prime}\|^2+1}}{\|x_*^{\prime}\|}\widetilde{\alpha}_1\epsilon.
	\end{equation}
Let $\epsilon_t = \min\left\{\epsilon_1,\frac{\|x_*^{\prime}\|}
{2\sqrt{2}\widetilde{\alpha}_1\sqrt{\|x_*^{\prime}\|^2+1}}\right\}$
and $\kappa_t=2\sqrt{2}\frac{\sqrt{\|x_*^{\prime}\|^2+1}}
{\|x_*^{\prime}\|}\widetilde{\alpha}_1$. Then we reach the
bound~\eqref{eq:VkV*}.
\end{proof}

\subsection{Properties of the 2DRQ}
We now investigate the properties of 
the $2\times2$ 2DRQ $(A_k, C_k)$. We show that $C_k$ is an indefinite matrix and the 2D-eigenvalues 
of $(A_k, C_k)$ defined in \eqref{eq:akck} are simple.

\begin{lemma}\label{Thm:behavior2}
Let $\epsilon = \max\{|\mu_k-\mu_*|, |\lambda_k-\lambda_*|, 
\dist(x_k,\mathcal{X}_*)\}$. 
There exists $\epsilon_2>0$ depending on $(A,C,\mu_*,\lambda_*)$ 
such that if $\epsilon\leq\epsilon_2$, 
\begin{equation} \label{eq:a12bd} 
|a_{12,k}|\geq\frac{1}{2}|a_{12,*}|,
\end{equation} 
where $a_{ij,k}$ and $a_{ij,*}$ denote the $(i,j)$-elements of 
$A_k$ and $A_*$ respectively.
\end{lemma} 
\begin{proof} 
By~\Cref{lemmaRQItypeI}, there exists positive constants 
$\epsilon_t$ and $\kappa_t$ depending on $(A,C,\mu_*,\lambda_*)$, 
such that if $\epsilon\leq\epsilon_t$, we have
\begin{equation}\label{eq:thmbehavior2eqs0}
\|\sin\Theta(V_*,V_k)\|\leq \kappa_t\epsilon.
\end{equation}
Since $V_*^HCV_*=\Diag(c_{1,*},c_{2,*})$ and 
$V_k^HCV_k=\Diag(c_{1,k},c_{2,k})$,
by \Cref{Thm:doubleV}, there exists positive constants 
$t_0,\kappa_0$ depending on $(C,c_{1,*},c_{2,*})$ such that 
if $\kappa_t\epsilon\leq t_0$, i.e., $\epsilon\leq\frac{t_0}{\kappa_t}$, 
there exists $\gamma_1,\gamma_2$ with absolute value 1, satisfying
\begin{equation}\label{eq:thmbehavior2eqs10}
\|V_*-V_k\Diag(\gamma_1,\gamma_2)\|\leq \kappa_0\|\sin\Theta(V_*,V_k)\|\leq \kappa_0\kappa_t\epsilon,
\end{equation}
where $t_0,\kappa_0$ are constants defined in \Cref{Thm:doubleV} 
and only depend on $(C,c_{1,*}-c_{2,*})$. 
	
Denote $E= {\begin{bmatrix} E_1&E_2 \end{bmatrix}}
\equiv V_*-V_k\Diag(\gamma_1,\gamma_2)$, 
and write $V_k = \begin{bmatrix} v_1&v_2 \end{bmatrix}$. 
Utilizing \eqref{eq:thmbehavior2eqs10}, we have
\begin{align*}
|a_{12,k}|&=|v_1^HAv_2| 
=|\overline{\gamma}_1v_1^HAv_2\gamma_2|
=|(v_{1,*}-E_1)^HA(v_{2,*}-E_2)| \\ 
&\geq |v_{1,*}^HAv_{2,*}|-|v_{1,*}^HAE_2|-|E_1^HAv_2|
\geq |a_{12,*}|-2\|E\|\|A\|
\geq |a_{12,*}|-2\kappa_0\kappa_t\|A\|\epsilon.
\end{align*}
Therefore, if $\epsilon\leq\epsilon_2\equiv \min\left\{\epsilon_t,\frac{t_0}{\kappa_t},\frac{|a_{12,*}|}{4\kappa_0\kappa_t\|A\|} \right\}$, we have 
$|a_{12,k}|\geq\frac{|a_{12,*}|}{2}$. Since $\epsilon_2$ only depends on 
$(A,C,\mu_*,\lambda_*)$, we reach the inequality~\eqref{eq:a12bd}. 
\end{proof} 

Next we show that $C_k = V^H_k C V_k$ is an indefinite matrix, 
and its eigenvalues are bounded away from 0. 

\begin{lemma}\label{thm:Ckindef}
Let $\epsilon = \max\{|\mu_k-\mu_*|, |\lambda_k-\lambda_*|, 
\dist(x_k,\mathcal{X}_*)\}$. There exists $\epsilon_T>0$ depending 
on $(A,C,\mu_*,\lambda_*)$, such that if $\epsilon\leq\epsilon_T$, 
\begin{itemize}
\item[(i)] $C_k$ is indefinite.
\item[(ii)] $\frac{1}{2}c_{1,*}\leq\  c_{1,k}\leq\frac{3}{2}c_{1,*}$
and $\frac{3}{2}c_{2,*}\leq\  c_{2,k}\leq\frac{1}{2}c_{2,*}$. 
\end{itemize} 	
\end{lemma}
\begin{proof} 
In \Cref{lem:cstar}, we have proven that 
$C_* = V^H_* C V_* = \Diag(c_{1,*}, c_{2,*})$ is indefinite
with $c_{1,*} > 0 > c_{2,*}$. 
	
For $V_k$ defined in \eqref{eq:vkdef2} and \eqref{eq:determineVk}, 
by \Cref{lemma:Zhang}, there exists a unitary matrix $U_k$, such that 
\[
\|V_*-V_kU_k\|\leq\sqrt{2}\|\sin\Theta(V_k,V_*)\|.
\] 
According to Weyl's theorem~\cite[p.\,198]{demmel1997applied}, 
and the fact that $c_{1,k}\geq c_{2,k}$ are eigenvalues of $C_k$, we have
\begin{align*} 
|c_{i,*}-c_{i,k}| 
& \leq \|V_*^HCV_*-U_k^HV_k^HCV_kU_k\| \\ 
& \leq\|(V_*-V_kU_k)^HCV_*\|+
\|U_k^HV_k^HC(V_*-V_kU_k)\| \\
& \leq 2\sqrt{2}\|C\|\|\sin\Theta(V_k,V_*)\| 
\quad \mbox{for $i = 1, 2$}.  
\end{align*} 
Now by \Cref{lemmaRQItypeI}, there exists positive 
constants $\epsilon_t,\kappa_t$ depending on $(A,C,\mu_*,\lambda_*)$, 
such that if $\epsilon\leq\epsilon_t$, 
\[
\|\sin\Theta(V_k,V_*)\|\leq \kappa_t\epsilon. 
\] 
Therefore, if 
\[ 
\epsilon\leq\epsilon_T\equiv
\min\left\{\epsilon_t,\frac{\min\{|c_{1,*}|,|c_{2,*}|\}}{4\sqrt{2}\kappa_t\|C\|}\right\}, 		
\] 
we have 
\begin{align*}
\max\{|c_{1,*}-c_{1,k}|, |c_{2,*}-c_{2,k}|\}
&\leq 2\sqrt{2}\|C\|\|\sin\Theta(V_k,V_*)\|\\
&\leq \frac{\min\{|c_{1,*}|,|c_{2,*}|\}}{2}.
\end{align*}
Consequently, we have 
\begin{align*}
c_{2,k}&=c_{2,*}-(c_{2,*}-c_{2,k})\leq c_{2,*} + |c_{2,*}-c_{2,k}|
\leq  c_{2,*}+\frac{\min\{|c_{1,*}|,|c_{2,*}|\}}{2}\leq\frac{1}{2}c_{2,*};\\
c_{2,k}&=c_{2,*}-(c_{2,*}-c_{2,k})\geq c_{2,*}-|c_{2,*}-c_{2,k}|
\geq c_{2,*}-\frac{\min\{|c_{1,*}|,|c_{2,*}|\}}{2}\geq \frac{3}{2}c_{2,*}. 
\end{align*}
By similar argument,  
we have $\frac{1}{2}c_{1,*}\leq c_{1,k}\leq\frac{3}{2}c_{1,*}$. 
This completes the proof.
\end{proof} 

\subsection{Main result}
The following is the main result on 
the quadratic convergence of the 2DRQI to 
compute a nonsingular simple 2D-eigenvalue.

\begin{theorem} \label{Thm:convtype1}
Let $\epsilon = \max\{|\mu_k-\mu_*|, |\lambda_k-\lambda_*|, 
\dist(x_k,\mathcal{X}_*)\}$. There exists 
a constant $\epsilon_0>0$ depending on $(A,C,\mu_*,\lambda_*)$, 
such that if $\epsilon\leq\epsilon_0$, 
\begin{enumerate}[(i)] 
\item 
$k+1$ approximate 2D-eigentriplet 
$(\mu_{k+1},\lambda_{k+1},x_{k+1})$ is determined as in \eqref{eq:update1a}. 

\item The 2DRQI has locally quadratic convergence rate, i.e.,
\begin{equation} 
|\mu_{k+1}-\mu_*|\leq \kappa_1 \epsilon^2, \quad
|\lambda_{k+1}-\lambda_*| \leq \kappa_2 \epsilon^4 
\quad \mbox{and} \quad
\dist(x_{k+1},\mathcal{X}_*) \leq \kappa_3 \epsilon^2, 
\end{equation} 
where $\kappa_1, \kappa_2$ and $\kappa_3$ 
are constants depending on $(A, C, \mu_*, \lambda_*)$.   
\end{enumerate} 
\end{theorem}

The proof of \Cref{Thm:convtype1} is split into the following three 
lemmas.  The first lemma shows that $\Span\{V_k\}$ contains 
a second-order approximate vector to $\mathcal{X}_*$. 

\begin{lemma}\label{prop1}
Let 
$\epsilon = \max\{|\mu_k-\mu_*|, |\lambda_k-\lambda_*|, 
\dist(x_k,\mathcal{X}_*)\}$. If $\epsilon\leq\epsilon_1$, 
where $\epsilon_1$ is defined in \Cref{Thm:behavior}, 
there exists ${x}_{k+1}^{\rm (p)}\in \Span\{V_k\}$, 
such that $\|{x}_{k+1}^{\rm (p)}\|=1$ and
\begin{equation}
\dist({x}_{k+1}^{\rm (p)},\mathcal{X}_*)\leq
\frac{4(\|C\|+1)}{\sigma_{n}(\widehat{J}_*)}\,\epsilon^2.
\end{equation}
\end{lemma}
\begin{proof} 
Denote $\Delta \mu_k=\mu_*-\mu_k$, 
$\Delta \lambda_k=\lambda_*-\lambda_k$ and 
$\Delta x_k=x_*-x_k$. Then 
\[
|\Delta \mu_k|\leq \epsilon, \quad  
|\Delta \lambda_k|\leq \epsilon, \quad 
\|\Delta x_k\|\leq \epsilon,
\] 
and 
\begin{equation} \label{eq:Deltak}
(A-\mu_kC-\lambda_kI)x_*-\Delta\mu_kC(x_k+\Delta x_k)
-\Delta\lambda_k(x_k+\Delta x_k)=0.
\end{equation} 
Write \eqref{eq:Deltak} as 
\begin{equation} \label{eq:jk1}
\widehat{J}_k 
\left[\begin{array}{c}x_*\\ \Delta\mu_k\\
\Delta\lambda_k\end{array}\right]
=\Delta\mu_kC\Delta x_k+\Delta\lambda_k\Delta x_k.
\end{equation} 
Then the minimum norm solution of \eqref{eq:jk1} is given by 
$$
\begin{bmatrix}
\Delta x^{\rm (p)}\\ \Delta \mu^{\rm (p)} \\ \Delta \lambda^{\rm (p)}
\end{bmatrix} = 
V_k^{(J)}(\Sigma_{k}^{(J)})^{-1}(U_k^{(J)})^H
\left(\Delta\mu_kC\Delta x_k + \Delta\lambda_k\Delta x_k\right), 
$$
where $\widehat{J}_k = U_k^{(J)}\Sigma_{k}^{(J)}(V_k^{(J)})^H$ is 
the SVD of $\widehat{J}_k$, 
$U_k^{(J)}\in\mathbb{C}^{n\times n}$, $\Sigma_{k}^{(J)}\in\mathbb{C}^{n\times n}$ and $V_k^{(J)}\in\mathbb{C}^{(n+2) \times n}$.  
Let 
\begin{equation}\label{eq:lepropeqs1}
\widehat{x}_{k+1}^{\rm (p)} = x_*-\Delta x^{\rm (p)},  \quad
\Delta\mu_{k+1}^{\rm (p)} = \Delta\mu_k-\Delta\mu^{\rm (p)}, \quad
\Delta\lambda_{k+1}^{\rm (p)} = \Delta\lambda_k-\Delta \lambda^{\rm (p)}.
\end{equation}
Then 
\begin{equation}\label{presolution2}
\widehat{J}_k \left[\begin{array}{c} \widehat{x}_{k+1}^{\rm (p)}\\
\Delta\mu_{k+1}^{\rm (p)} \\
\Delta\lambda_{k+1}^{\rm (p)}\end{array}\right]=0,
\end{equation}
and the vector $\begin{bmatrix}
\widehat{x}_{k+1}^{\rm (p)}\\
\Delta\mu_{k+1}^{\rm (p)} \\
\Delta\lambda_{k+1}^{\rm (p)}
\end{bmatrix}$ belongs to the nullspace of 
$\widehat{J}_k$ and thus $\widehat{x}_{k+1}^{\rm (p)}\in\Span\{V_k\}$.  
The desired vector $x_{k+1}^{\rm (p)}$ is then given by  
$$
x_{k+1}^{\rm (p)} 
= \widehat{x}_{k+1}^{\rm (p)}/\|\widehat{x}_{k+1}^{\rm (p)}\|.
$$ 
Obviously, ${x}_{k+1}^{\rm (p)}\in\Span\{V_k\}$ and $\|x_{k+1}^{\rm (p)}\|=1$. 
The approximation error satisfies:  
\begin{align}
\|x_*-x_{k+1}^{\rm (p)}\| 
& \leq \|x_*-\widehat{x}_{k+1}^{\rm (p)}\|+
       \|\widehat{x}_{k+1}^{\rm (p)}-x_{k+1}^{\rm (p)}\| \nonumber \\
& =\|\Delta x^{\rm (p)}\|+\left|\|\widehat{x}_{k+1}^{\rm (p)}\|-1\right| 
       \nonumber \\ 
& \leq 2\|\Delta x^{\rm (p)}\| \label{eq:leprop1eqs2} \\ 
& \leq\frac{4(\|C\|+1)}{\sigma_{n}(\widehat{J}_*)} \epsilon^2,\label{eq:leprop1eqs3}
\end{align}
where in \eqref{eq:leprop1eqs2}, we use the definition $\widehat{x}_{k+1}^{\rm (p)} = x_*-\Delta x^{\rm (p)}$ and the following fact: 
\[
1-\|\Delta x^{\rm (p)}\|\leq \|x_*-\Delta x^{\rm (p)}\| 
\leq 1+\|\Delta x^{\rm (p)}\|.
\] 
For the inequality \eqref{eq:leprop1eqs3}, we use the fact that 
the minimum norm solution of \eqref{eq:jk1} satisfies
\begin{equation} \label{eq:xpbd} 
\|\Delta x^{\rm (p)}\|
\leq \frac{(\|C\|+1)}{\sigma_{n}(\widehat{J}_k)} \epsilon^2
\leq \frac{2(\|C\|+1)}{\sigma_{n}(\widehat{J}_*)} \epsilon^2,
\end{equation} 
where the second inequality is due to \Cref{Thm:behavior}.
This completes the proof. 
\end{proof} 

We call the vector $x_{k+1}^{\rm (p)}$ in \Cref{prop1} a 
``pre-optimal'' solution since $x_{k+1}^{{\rm (p)}H}Cx_{k+1}^{\rm (p)}$ 
does not necessarily vanish and is thus generally not 
the $k+1$-th iterate $x_{k+1}$. 
The next lemma shows that based on this pre-optimal solution, 
we can construct a vector $\widetilde{x}_{k+1}$ in $\Span\{V_k\}$
satisfying \eqref{eq:1b}. 

\begin{lemma}\label{prop2}
Let $\epsilon = \max\{|\mu_k-\mu_*|, |\lambda_k-\lambda_*|, \dist(x_k,\mathcal{X}_*)\}$. If $\epsilon\leq\min\{\epsilon_1,\epsilon_T\}$, where 
$\epsilon_1, \epsilon_T$ are defined in \Cref{Thm:behavior,thm:Ckindef}, 
then there exists
\begin{equation} \label{eq:xkdef} 
\widetilde{x}_{k+1}\in
\mathcal{X}_k \equiv \{x \mid x \in \Span\{V_k\}, x^HCx = 0, x^Hx = 1 \},
\end{equation} 
such that
\begin{equation} \label{eq:xk1t} 
\dist(\widetilde{x}_{k+1}, \mathcal{X}_*) \leq 
\left(\frac{16\|C\|}{\sqrt{-c_{1,*}c_{2,*}}}+4\right)
\frac{\|C\|+1}{\sigma_{n}(\widehat{J}_*)}\, \epsilon^2,
\end{equation}
\end{lemma}
\begin{proof} We prove by construction. 
First, for $x_{k+1}^{\rm (p)}$ determined in \Cref{prop1}, we have 
$x^{\rm (p)}_{k+1} \in \Span\{V_k\}$. 
Write $x_{k+1}^{\rm (p)}=t_{p}v_1+s_{p}v_2$ with $|t_p|^2+|s_p|^2=1$, 
where $V_k = \begin{bmatrix} v_1&v_2 \end{bmatrix}$. 

Let $\widetilde{x}_*$ be the vector 
in $\mathcal{X}_*$ closest to $x_{k+1}^{\rm (p)}$. Then we have
\begin{equation} \label{eq:xcxk}
\begin{aligned}
|x_{k+1}^{{\rm (p)}H}Cx_{k+1}^{\rm (p)}|
&=|(x_{k+1}^{\rm (p)}-\widetilde{x}_*)^HCx_{k+1}^{\rm (p)}+\widetilde{x}_*^HC(x_{k+1}^{\rm (p)}-\widetilde{x}_*)+\widetilde{x}_*^HC\widetilde{x}_*|\\
&= |(x_{k+1}^{\rm (p)}-\widetilde{x}_*)^HCx_{k+1}^{\rm (p)}
    +\widetilde{x}_*^HC(x_{k+1}^{\rm (p)}-\widetilde{x}_*)|\\
&\leq 2\|C\|\|x_{k+1}^{\rm (p)}-\widetilde{x}_*\|.
\end{aligned}
\end{equation} 
Define
\begin{equation}\label{eq:tsdef}
\widetilde{x}_{k+1}=tv_1+sv_2,
\end{equation} 
where 
\begin{equation}\label{eq:definets}
t = \sign(t_p)\sqrt{\frac{-c_{2,k}}{c_{1,k}-c_{2,k}}},\quad 
s = \sign(s_p)\sqrt{\frac{c_{1,k}}{c_{1,k}-c_{2,k}}}.
\end{equation}
Obviously, $\|\widetilde{x}_{k+1}\|=1$ and 
$\widetilde{x}_{k+1}^HC\widetilde{x}_{k+1}=0$. 
This implies $\widetilde{x}_{k+1}\in\mathcal{X}_k$, and 
\begin{equation}\label{eq:xxpdiff}
\|\widetilde{x}_{k+1} - x_{k+1}^{\rm (p)}\| 
= \|(t-t_p)v_1 + (s-s_p)v_2\| 
=  \sqrt{||t|-|t_p||^2 + ||s|-|s_p||^2}.
\end{equation}
where the last equality results from $v_1\perp v_2$ 
and the sign of $t,s$.

Next we derive an upper bound of the right-hand-side of
equation~\eqref{eq:xxpdiff}.   Note that
\begin{align*}
x_{k+1}^{{\rm (p)}H}Cx_{k+1}^{\rm (p)} & = |t_p|^2c_{1,k}+(1-|t_p|^2)c_{2,k}, \\
0=\widetilde{x}_{k+1}^H C \widetilde{x}_{k+1} & = |t|^2c_{1,k}+(1-|t|^2)c_{2,k};
\end{align*}
Substracting the previous two equations, we have
\[
-x_{k+1}^{{\rm (p)}H}Cx_{k+1}^{\rm (p)} = 
(|t|^2-|t_p|^2)(c_{1,k}-c_{2,k}) =  
(|t|-|t_p|)(|t|+|t_p|)(c_{1,k}-c_{2,k}).
\]
Similarly, in terms of $s$, 
\[
x_{k+1}^{{\rm (p)}H}Cx_{k+1}^{\rm (p)} = 
(|s|-|s_p|)(|s|+|s_p|)(c_{1,k}-c_{2,k}).
\]
Therefore, 
\begin{equation}\label{eq:lemmaprop2eqs1}
\begin{aligned}
\sqrt{||t|-|t_p||^2 + ||s|-|s_p||^2}&=\sqrt{\left(\frac{1}{|t|+|t_p|}\right)^2+\left(\frac{1}{|s|+|s_p|}\right)^2}\frac{|x_{k+1}^{{\rm (p)}H}Cx_{k+1}^{\rm (p)}|}{c_{1,k}-c_{2,k}}\\
&\leq\sqrt{\frac{1}{|t|^2}+\frac{1}{|s|^2}}\frac{|x_{k+1}^{{\rm (p)}H}Cx_{k+1}^{\rm (p)}|}{c_{1,k}-c_{2,k}}\\
&=\sqrt{\frac{c_{1,k}-c_{2,k}}{-c_{2,k}}+\frac{c_{1,k}-c_{2,k}}{c_{1,k}}}\frac{|x_{k+1}^{{\rm (p)}H}Cx_{k+1}^{\rm (p)}|}{c_{1,k}-c_{2,k}}\\
&=\frac{|x_{k+1}^{{\rm (p)}H}Cx_{k+1}^{\rm (p)}|}{\sqrt{-c_{1,k}c_{2,k}}}
\end{aligned}
\end{equation}
By \Cref{thm:Ckindef}, 
\begin{equation}\label{eq:prop2:1}
\sqrt{-c_{1,k}c_{2,k}}\geq \frac{1}{2}\sqrt{-c_{1,*}c_{2,*}}.
\end{equation}
Thus from \eqref{eq:xcxk} and \eqref{eq:prop2:1}, we have
\begin{equation}\label{eq:prop2:2}
\sqrt{||t|-|t_p||^2 + ||s|-|s_p||^2} \leq 
\frac{4\|C\|\|x_{k+1}^{\rm (p)}-\widetilde{x}_*\|}{\sqrt{-c_{1,*}c_{2,*}}}.
\end{equation}
Finally, by \Cref{prop1}, we have 
\begin{equation}\label{eq:lemmaprop2eqs2}
\dist(\widetilde{x}_{k+1},\mathcal{X}_*)
\leq \|\widetilde{x}_{k+1}-x_{k+1}^{\rm (p)}\| +\dist(x_{k+1}^{\rm (p)},\mathcal{X}_*)
\leq \left(\frac{16\|C\|}{\sqrt{-c_{1,*}c_{2,*}}}+4\right)\frac{\|C\|+1}{\sigma_{n}(\widehat{J}_*)}\epsilon^2.
\end{equation}
This completes the proof.  
\end{proof} 

The next result shows there exists a 2D-Ritz triplet sufficiently 
close to the target 2D-eigentriplet.

\begin{lemma}\label{lemma:xkp}
Let $\epsilon = \max\{|\mu_k-\mu_*|, |\lambda_k-\lambda_*|,
\dist(x_k,\mathcal{X}_*)\}$. Then if 
$\epsilon\leq\min\{\epsilon_1,\epsilon_2,\epsilon_T\}$, 
where $\epsilon_1,\epsilon_2,\epsilon_T$ are defined in 
\Cref{Thm:behavior,Thm:behavior2,thm:Ckindef}, respectively, we have
\begin{enumerate}[(i)]  
\item 2DRQ $(A_k,C_k)$ has two 2D Ritz values 
$(\nu_{k,j},\theta_{k,j})$ for $j=1, 2$.

\item 
2DRQ $(A_k,C_k)$ has at least one of 2D Ritz triplets 
$(\nu_{k,j},\theta_{k,j},x_{k,j})$ for $j=1, 2$ 
satisfying 
\begin{equation} \label{eq:Ritzbd} 
|\nu_{k,j}-\mu_*|\leq \kappa_1\epsilon^2, \quad
|\theta_{k,j}-\lambda_*|\leq \kappa_2\epsilon^4,  \quad
\dist(x_{k,j},\mathcal{X}_*) \leq \kappa_3 \epsilon^2,
\end{equation} 
where $\kappa_1,\kappa_2,\kappa_3>0$ are constants depending on 
$(A,C,\mu_*,\lambda_*)$.  
\end{enumerate} 
\end{lemma}
\begin{proof} 
Let $\epsilon\leq\min\{\epsilon_1,\epsilon_2,\epsilon_T\}$. 
By \Cref{Thm:behavior2}, $a_{12,k} \equiv (V^H_k A V_k)_{1,2} \neq 0$. 
Thus the $2\times2$ 2DEVP~\eqref{projectEVP} has two 2D-eigentriplets 
(see \Cref{sec:rqi}): 	
\[
(\nu(\alpha_{k,j}), \theta(\alpha_{k,j}), z(\alpha_{k,j})), 
\quad j = 1,2,
\]
where $\nu(\cdot),\theta(\cdot), z(\cdot)$ 
are defined in \eqref{eq:mulamalpha}, and
\[
\alpha_{k,1} = \frac{\overline{a}_{12,k}}{|a_{12,k}|},\quad 
\alpha_{k,2} = -\frac{\overline{a}_{12,k}}{|a_{12,k}|}.
\]
This proves the result(i). 

We note consider the result (ii). 
First, since $\epsilon\leq\min\{\epsilon_1, \epsilon_T\}$, 
by \Cref{prop2},  
$\widetilde{x}_{k+1}=tv_1+sv_2\in \mathcal{X}_k$ 
is well defined with $t,s$ defined by \eqref{eq:tsdef} and
$V_k = \begin{bmatrix} v_1&v_2 \end{bmatrix}$.
For brevity, let $\widetilde{x}_{k+1} := \overline{\sign(t)}\widetilde{x}_{k+1}$. 
Then $\dist(\widetilde{x}_{k+1},\mathcal{X}_*)$ is invariant. 
By equations \eqref{eq:definets} and \eqref{eq:mulamalpha}, 
$\widetilde{x}_{k+1}$ can be further written as
\begin{equation} \label{eq:xvzexpress} 
\widetilde{x}_{k+1} = V_kz(\widetilde{\alpha}), \quad |\widetilde{\alpha}|=1.
\end{equation}
Note that the 2D Ritz triplets are given by 
\[
(\nu_{k,i}, \theta_{k,i}, x_{k,i}) =   
(\nu(\alpha_{k,i}), \theta(\alpha_{k,i}), V_kz(\alpha_{k,i})) 
\quad \mbox{for $i=1,2$}.
\]

We next prove that one of $\alpha_{k,i}$ is a good approximation 
to $\widetilde{\alpha}$. 
We begin by showing that the imaginary part of $\nu(\widetilde{\alpha})$ 
is small by utilizing $\widetilde{x}_{k+1}$ is close to $\mathcal{X}_*$.
Let $\Delta x = \widetilde{x}_*-\widetilde{x}_{k+1}$, where 
$\widetilde{x}_*$ is the vector in $\mathcal{X}_*$ closest 
to $\widetilde{x}_{k+1}$. 
Starting with \eqref{eq:mulamalpha}, we have
\begin{align}
\nu(\widetilde{\alpha})
&=\widetilde{x}_{k+1}^{H}CV_kV_k^HA\widetilde{x}_{k+1}/
\|C_kz(\widetilde{\alpha})\|^2 \quad \mbox{(due to \eqref{eq:xvzexpress})} 
\nonumber \\
&=\widetilde{x}_{k+1}^{H}CV_kV_k^HA(\widetilde{x}_*-\Delta x)/
\|C_kz(\widetilde{\alpha})\|^2  \nonumber \\
&=\widetilde{x}_{k+1}^{H}CV_kV_k^HA\widetilde{x}_*/
\|C_kz(\widetilde{\alpha})\|^2+h_1 \quad \mbox{($h_1$ denotes remaining terms)} 
\nonumber \\
&=\mu_*\widetilde{x}_{k+1}^{H}CV_kV_k^HC\widetilde{x}_*/
\|C_kz(\widetilde{\alpha})\|^2+\lambda_*\widetilde{x}_{k+1}^{H}CV_kV_k^H\widetilde{x}_*/
\|C_kz(\widetilde{\alpha})\|^2+h_1 \nonumber \\
&=\mu_*\widetilde{x}_{k+1}^{H}CV_kV_k^HC\widetilde{x}_{k+1}/
\|C_kz(\widetilde{\alpha})\|^2+\lambda_*\widetilde{x}_{k+1}^{H}CV_kV_k^H\widetilde{x}_{k+1}/
\|C_kz(\widetilde{\alpha})\|^2+h_2 \nonumber \\
&= \mu_*\widetilde{x}_{k+1}^{H}CV_kV_k^HC\widetilde{x}_{k+1}/
\|C_kz(\widetilde{\alpha})\|^2+\lambda_*\widetilde{x}_{k+1}^{H}C\widetilde{x}_{k+1}/
\|C_kz(\widetilde{\alpha})\|^2+h_2 \nonumber \\
&=\mu_* \widetilde{x}_{k+1}^{H}CV_kV_k^HC\widetilde{x}_{k+1}/
\|C_kz(\widetilde{\alpha})\|^2+h_2, \label{eq:mutildealpha}
\end{align}
where
\[
h_2 = -\frac{\widetilde{x}_{k+1}^HCV_kV_k^HA\Delta x}{\|C_kz(\widetilde{\alpha})\|^2}+\frac{\mu_*\widetilde{x}_{k+1}^HCV_kV_k^HC\Delta x}{\|C_kz(\widetilde{\alpha})\|^2}+\frac{\lambda_*\widetilde{x}_{k+1}^HCV_kV_k^H\Delta x}{\|C_kz(\widetilde{\alpha})\|^2}
\] 
and the 6-th equality comes from $\widetilde{x}_{k+1}\in\Span \{V_k\}$. 
To estimate $h_2$, note that
\begin{equation}\label{eq:severalIdentities}
\begin{aligned} 
\|\widetilde{x}_{k+1}^HCV_k\| &=\|z(\widetilde{\alpha})^HV_k^HCV_k\|
=\|z(\widetilde{\alpha})^HC_k\| \\
\|C_kz(\alpha)\|& =\sqrt{-c_{1,k}c_{2,k}},\quad \forall \alpha: |\alpha|=1. 
\end{aligned} 
\end{equation}  
Then
\[
\begin{aligned}
|h_2| &=\left|-\frac{\widetilde{x}_{k+1}^HCV_kV_k^H(A-\mu_*C-\lambda_*I)\Delta x}{\|C_kz(\widetilde{\alpha})\|^2}\right|\\
&\leq \frac{\|A-\mu_*C-\lambda_*I\|}{\|C_kz(\widetilde{\alpha})\|}\|\Delta x\|
= \frac{\|A-\mu_*C-\lambda_*I\|}{\sqrt{-c_{1,k}c_{2,k}}}\|\Delta x\|.
\end{aligned}
\]
Since $\mu_* \widetilde{x}_{k+1}^{H}CV_kV_k^HC\widetilde{x}_{k+1}/\|C_kz(\alpha)\|^2$ is real, \eqref{eq:mutildealpha} implies
	\begin{equation}\label{eq:imagmutilde}
		|\imag\nu(\widetilde{\alpha})|\leq \frac{\|A-\mu_*C-\lambda_*I\|}{\sqrt{-c_{1,k}c_{2,k}}}\|\Delta x\|.
	\end{equation}
	On the other hand, by \eqref{eq:mulamalpha}, straight calculation shows
	\begin{equation}\label{eq:imagmu}
		\imag\nu(\widetilde{\alpha}) = \imag\frac{c_{1,k}\widetilde{\alpha} a_{12,k}+c_{2,k}\overline{\widetilde{\alpha} a_{12,k}}}{(c_{1,k}-c_{2,k})\sqrt{-c_{1,k}c_{2,k}}}=\imag\frac{\widetilde{\alpha} a_{12,k}}{\sqrt{-c_{1,k}c_{2,k}}}.
	\end{equation}		
By \eqref{eq:imagmutilde} and \eqref{eq:imagmu}, $\widetilde{\alpha}$ satisfies
\begin{equation}\label{eq:alphaa12}
|\imag (\widetilde{\alpha}a_{12,k})|\leq\|A-\mu_*C-\lambda_*I\|\|\Delta x\|.
\end{equation}
Let $e^{\ii\theta_0}=\frac{\overline{a}_{12,k}}{|a_{12,k}|}$ and $\widetilde{\alpha}=e^{\ii\widetilde{\theta}}e^{\ii\theta_0}$, where $\theta_0,\widetilde{\theta}\in(-\pi,\pi]$. Then \eqref{eq:alphaa12} turns to 
	\[
	|\sin\widetilde{\theta}||a_{12,k}|\leq \|A-\mu_*C-\lambda_*I\|\|\Delta x\|.
	\]
Therefore,
\begin{equation}\label{eq:sinTildeTheta}
|\sin\widetilde{\theta}|\leq 
\frac{\|A-\mu_*C-\lambda_*I\|}{|a_{12,k}|}\|\Delta x\|.
\end{equation}
When $\cos\widetilde{\theta}>0$, we let $j=1$; otherwise we let $j=2$. 
Utilizing expressions $\alpha_{k,i}=\pm e^{\ii \theta_0}, i=1,2$, 
$\widetilde{\alpha}=e^{\ii\widetilde{\theta}}e^{\ii \theta_0}$ 
and \eqref{eq:sinTildeTheta}, for the chosen $j$, we have
\begin{align}
|\alpha_{k,j}-\widetilde{\alpha}|
&=\left||\cos\widetilde{\theta}|-1+\ii \sin\widetilde{\theta}\right| \\ 
& \leq\sqrt{2-2|\cos{\widetilde{\theta}}|}
\leq \sqrt{2-2\cos^2{\widetilde{\theta}}}\\
& = \sqrt{2}|\sin{\widetilde{\theta}}|
\leq \frac{\sqrt{2}\|A-\mu_*C-\lambda_*I\|}{|a_{12,k}|}\|\Delta x\|.
\label{eq:lemmaxkpeqs1}
\end{align}
This implies $\alpha_{k,j}$ is a good approximation to $\widetilde{\alpha}$. 
	
We now consider the distance between $(\nu_{k,j},\theta_{k,j},x_{k,j})$
and $(\mu_*,\lambda_*,\mathcal{X}_*)$  for the chosen $j$.
Straight calculation shows
\begin{equation}\label{eq:xkiminusxkplus}
\begin{aligned}
\|x_{k,j}-\widetilde{x}_{k+1}\|&=\|V_kz(\alpha_{k,j})-V_kz(\widetilde{\alpha})\|\\
&=\|z(\alpha_{k,j})-z(\widetilde{\alpha})\|\\
&=\frac{\sqrt{c_{1,k}}}{\sqrt{c_{1,k}-c_{2,k}}}|\alpha_{k,j}-\widetilde{\alpha}|\\
&\leq |\alpha_{k,j}-\widetilde{\alpha}|\\
&\leq \sqrt{2}\frac{\|A-\mu_*C-\lambda_*I\|}{|a_{12,k}|}\|\Delta x\|.
\end{aligned}
\end{equation}
Then according to \eqref{eq:xkiminusxkplus}, \Cref{thm:Ckindef,prop2}, 
we immediately have
\begin{align}
\dist(x_{k,j},\mathcal{X}_*)&\leq \|x_{k,j}-\widetilde{x}_{k+1}\|+\dist(\widetilde{x}_{k+1},\mathcal{X}_*) \nonumber \\
&\leq \sqrt{2}\frac{\|A-\mu_*C-\lambda_*I\|}{|a_{12,k}|}\|\Delta x\|+\|\Delta x\| \label{eq:lemmaxkpeqs3}\\
&\leq 2\sqrt{2}\frac{\|A-\mu_*C-\lambda_*I\|}{|a_{12,*}|}\|\Delta x\|+\|\Delta x\|\nonumber \\
&=\left(2\sqrt{2}\frac{\|A-\mu_*C-\lambda_*I\|}{|a_{12,*}|}+1\right)\left(\frac{16\|C\|}{\sqrt{-c_{1,*}c_{2,*}}}+4\right)\frac{\|C\|+1}{\sigma_{n}(\widehat{J}_*)}\epsilon^2 := \kappa_3 \epsilon^2. \label{eq:lemmaxkpeqs2}
\end{align}
Thus we have proved the result for approximate 2D-eigenvector $x_{k,j}$.
	
Now denote $\widehat{x}_*$ as the vector in $\mathcal{X}_*$ closest to $x_{k,j}$. 
The equation $(A-\mu_*C-\lambda_*I)\widehat{x}_*=0$ can be rewritten as    
\[ 
(A-\mu_*C-\lambda_*I)x_{k,j}= (A-\mu_*C-\lambda_*I)(x_{k,j}-\widehat{x}_*). 	
\] 
Multiplying $V_k^H$ on the left and utilizing 
$x_{k,j} = V_k z(\alpha_{k,j})$, we have		
\begin{equation}\label{eq:subprobres}
(A_k-\mu_*C_k-\lambda_*I)z(\alpha_{k,j})
=V_k^H(A-\mu_*C-\lambda_*I)(x_{k,j}-\widehat{x}_*) \equiv r.
\end{equation}
and 
\begin{equation}\label{eq:muproject}
\|r\| \leq \|A-\mu_*C-\lambda_*I\|\|x_{k,j}-\widehat{x}_*\|
\leq \|A-\mu_*C-\lambda_*I\|\kappa_3\epsilon^2.
\end{equation}
	
We now show that $\nu_{k,j}$ is a good approximation to $\mu_*$. 
Multiplying $z(\alpha_{k,j})^HC_k$ on the left of \eqref{eq:subprobres} 
and using $z(\alpha_{k,j})^HC_kz(\alpha_{k,j})=0$, we have
\begin{equation}\label{eq:lemmaxkpseqs5}
\mu_*=\frac{z(\alpha_{k,j})^HC_kA_kz(\alpha_{k,j})-z(\alpha_{k,j})^HC_kr}
           {\|C_kz(\alpha_{k,j})\|^2}
=\nu_{k,j} -\frac{z(\alpha_{k,j})^HC_kr}{\|C_kz(\alpha_{k,j})\|^2}.
\end{equation}
By \Cref{thm:Ckindef}, \eqref{eq:severalIdentities} and \eqref{eq:muproject}, 
\begin{align}
\frac{\|z(\alpha_{k,j})^HC_kr\|}{\|C_kz(\alpha_{k,j})\|^2}
&\leq \frac{\|r\|}{\|C_kz(\alpha_{k,j})\|} 
=\frac{\|r\|}{\sqrt{-c_{1,k}c_{2,k}}} \nonumber \\
&\leq\frac{2\|r\|}{\sqrt{-c_{1,*}c_{2,*}}} 
\leq\frac{2\|A-\mu_*C-\lambda_*I\|\kappa_3\epsilon^2}{\sqrt{-c_{1,*}c_{2,*}}} \
\equiv \kappa_1\epsilon^2. \label{eq:lemmaxkpseqs6}
\end{align}
Thus we have proved the desired error bound for $\nu_{k,j}$: 
\begin{equation} \label{eq:nubd} 
|\nu_{k,j}-\mu_*|\leq \kappa_1\epsilon^2.
\end{equation} 
	
Finally, we show $\theta_{k,j} = \theta(\alpha_{k,j})$ is a good 
approximation to $\lambda_*$. 
By the definition of $\theta(\alpha)$, 
\[  
\theta_{k,j}  = \theta(\alpha_{k,j})  
=z^H(\alpha_{k,j})A_kz(\alpha_{k,j}) 
 = x_{k,j}^HAx_{k,j}.
\]
Since $x^H_{k,j} C x_{k,j} = 0$ and $x_{k,j}^Hx_{k,j}=1$, 
by \Cref{Thm:lsq},  we have
\begin{equation}\label{eq:lamboundstype1}
|\theta_{k,j} - \lambda_*| \leq 
\|A-\mu_*C-\lambda_*I\|\|\widehat{x}_*-x_{k,j}\|^2 
\leq \|A-\mu_*C-\lambda_*I\|\kappa_3^2\epsilon^4.
\end{equation}
In summary, by \eqref{eq:lemmaxkpeqs2}, \eqref{eq:nubd} 
and \eqref{eq:lamboundstype1}, we have the bounds \eqref{eq:Ritzbd} 
when 
$\epsilon\leq\min\{\epsilon_1,\epsilon_2,\epsilon_T\}$. 
\end{proof} 

\paragraph{Proof of \Cref{Thm:convtype1}.}  
With \Cref{lemma:xkp}, {the $k+1$ approximate 2D-eigentriplet 
$(\mu_{k+1},\lambda_{k+1},x_{k+1})$ will be determined 
as in \eqref{eq:update1a}}. We only need to prove we will choose 
the desired 2D-Ritz triplet. Without loss of generality, 
we assume $j=1$ in \Cref{lemma:xkp}. 
	
Let $\epsilon\leq\min\{\epsilon_1,\epsilon_2,\epsilon_T\}$, 
where $\epsilon_1,\epsilon_2,\epsilon_T$ are defined in 
\Cref{Thm:behavior,Thm:behavior2,thm:Ckindef}, respectively. Then we have
\begin{equation}
\begin{aligned}
|\nu_{k,1}-\mu_{k}|&\leq |\nu_{k,1}-\mu_*|+|\mu_*-\mu_{k}|\leq \epsilon+\kappa_1\epsilon^2,\\
|\theta_{k,1}-\lambda_{k}|&\leq |\theta_{k,1}-\lambda_*|+|\lambda_*-\lambda_{k}|\leq \epsilon+\kappa_2\epsilon^4,
\end{aligned}
\end{equation}
On the other hand, straight calculation shows 
\begin{equation}
\begin{aligned}
|\theta_{k,1}-\theta_{k,2}| 
&= |z(\alpha_{k,1})^HA_kz(\alpha_{k,1})-z(\alpha_{k,2})^HA_kz(\alpha_{k,2})|\\
& = \frac{4|a_{12,k}|\sqrt{-c_{1,k}c_{2,k}}}{c_{1,k}-c_{2,k}}
\geq \frac{2|a_{12,*}|\sqrt{-c_{1,*}c_{2,*}}}{3(c_{1,*}-c_{2,*})}.
\end{aligned}
\end{equation}
Therefore,
\[
|\nu_{k,1}-\mu_{k}|+|\theta_{k,1}-\lambda_{k}|
\leq  2\epsilon+(\kappa_1+\kappa_2\epsilon^2)\epsilon^2
\]
and
\[
|\nu_{k,2}-\mu_{k}|+|\theta_{k,2}-\lambda_{k}|
\geq |\theta_{k,2}-\theta_{k,1}|-|\theta_{k,1}-\lambda_{k}|
\geq \frac{2|a_{12,*}|\sqrt{-c_{1,*}c_{2,*}}}
          {3(c_{1,*}-c_{2,*})}-\epsilon-\kappa_2\epsilon^4.
\]
Denote $T\equiv\frac{2|a_{12,*}|\sqrt{-c_{1,*}c_{2,*}}}{3(c_{1,*}-c_{2,*})}$. 
When 
\[
2\epsilon+(\kappa_1+\kappa_2\epsilon^2)\epsilon^2
\leq \frac{T}{2}-\epsilon-\kappa_2\epsilon^4,
\]
for example, let $\epsilon\leq\min\left\{\frac{T}{8},
\sqrt{\frac{T}{8(\kappa_1+2\kappa_2)}},1\right\}$, the 
following strict inequality holds 
\[
|\nu_{k,1}-\mu_{k}|+|\theta_{k,1}-\lambda_{k}|
< |\nu_{k,2}-\mu_{k}|+|\theta_{k,2}-\lambda_{k}|. 
\]	
Thus we will choose $(\nu_{k,1},\theta_{k,1},x_{k,1})$ as 
the next update $(\mu_{k+1},\lambda_{k+1},x_{k+1})$. 
	
Combining with \Cref{lemma:xkp}, the theorem is proven by 
letting $\epsilon_0=\min\left\{\epsilon_1,\epsilon_2,\epsilon_T,\frac{T}{8},
\sqrt{\frac{T}{8(\kappa_1+2\kappa_2)}},1\right\}$. 
\hfill $\Box$


\section{Convergence analysis for nonsingular multiple 2D-eigentriplets}
\label{sec:converge_analysis_II}

In this section, we prove that the 2DRQI is locally quadratically convergent for computing a nonsingular multiple 2D-eigentriplet.

\subsection{Properties of $J_k$}   

Let $(\mu_k, \lambda_k, x_k)$ be the $k$-th iterate to a nonsingular multiple 2D-eigentriplet $(\mu_*,\lambda_*,x_*)$, where $x_*$ is the vector in $\mathcal{X}_*$ closest to $x_k$ and $\mathcal{X}_*$ is the set of 2D eigenvectors defined in \eqref{eq:XstarStructureII}. The following lemma shows that when $(\mu_k,\lambda_k,x_k)$ is sufficiently close to 
$(\mu_*,\lambda_*,\mathcal{X}_*)$, Jacobian $J_k = J(\mu_k, \lambda_k, x_k)$ is nonsingular.

\begin{lemma}\label{Thm:behaviorII}
Assume $|\mu_k-\mu_*|\leq\epsilon^2$, $|\lambda_k-\lambda_*|\leq\epsilon^2$, 
and $\dist(x_k,\mathcal{X}_*)\leq\epsilon$ for some $\epsilon\leq1$.
There exists $\epsilon_1>0$ depending on $(A,C,\mu_*,\lambda_*)$, 
such that if $\epsilon\leq\epsilon_1$,   
\begin{enumerate}[(i)] 
\item  $\sigma_{\min}(J_k) \geq\frac{1}{2} \sigma_{\min,J_*}$,

\item  $\sigma_{n}(\widehat{J}_k) \geq\frac{1}{2} \sigma_{n,\widehat{J}_*}$, 

\item $\sigma_{n}(\widehat{J}(\mu_*,\lambda_*,x_k)) \geq
\frac{1}{2} \sigma_{n,\widehat{J}_*}$,
\end{enumerate} 
where $\sigma_{\min,J_*}$ and $\sigma_{n,\widehat{J}_*}$ are defined 
in \Cref{sec:propmulti2D}.
\end{lemma}  
\begin{proof} 
The proof is similar to the proof of \Cref{Thm:behavior}.  
Note that $x_*$ is the vector in $\mathcal{X}_*$ closest to $x_k$. 
Denote $\widehat{J}_*=\widehat{J}(\mu_*,\lambda_*,x_*)$. Then we have
\[
\|\widehat{J}(\mu_*,\lambda_*,x_k)-\widehat{J}_*\|\leq\|\widehat{J}_k-\widehat{J}_*\|\leq\|J_k-J(\mu_*,\lambda_*,x_*)\|\leq 3(\|C\|+1)\epsilon,
\]
Let
\[
\epsilon_1=\frac{\min\{\sigma_{\min,J_*}, \sigma_{n,\widehat{J}_*}\}}{6(\|C\|+1)}.
\]
Then if $\epsilon\leq\epsilon_1$, 
by Weyl's theorem \cite[p.\,198]{demmel1997applied},
we have 
\begin{align*}
|\sigma_{\min}(J_k)-\sigma_{\min}(J(\mu_*,\lambda_*,x_*))|
&\leq\frac{\sigma_{\min,J_*}}{2},\\						
|\sigma_{n}(\widehat{J}_k)-\sigma_{n}(\widehat{J}_*)|
&\leq\frac{\sigma_{n,\widehat{J}_*}}{2},\\
|\sigma_{n}(\widehat{J}(\mu_*,\lambda_*,x_k))-\sigma_{n}(\widehat{J}_*)| 
&\leq\frac{\sigma_{n,\widehat{J}_*}}{2}.	
\end{align*}  
Thus
\[
\sigma_{\min}(J_k)
\geq \sigma_{\min}(J(\mu_*,\lambda_*,x_*))-\frac{\sigma_{\min,J_*}}{2}
\geq\frac{\sigma_{\min,J_*}}{2}>0,
\]
\[
\sigma_{n}(\widehat{J}_k)
\geq \sigma_{n}(\widehat{J}_*)-\frac{\sigma_{n,\widehat{J}_*}}{2}
\geq\frac{\sigma_{n,\widehat{J}_*}}{2}>0,
\]
and 
\[
\sigma_{n}(\widehat{J}(\mu_*,\lambda_*,x_k))
\geq \sigma_{n}(\widehat{J}_*)-\frac{\sigma_{n,\widehat{J}_*}}{2}
\geq\frac{\sigma_{n,\widehat{J}_*}}{2}>0.
\]
This completes the proof. 
\end{proof}

\subsection{Approximation of $V_k$ to $V_*$ and 
nonsingularity of $C_k$}

Next we show the approximation of $V_k$ to $V_*$ and nonsingularity of 
$C_k = V^H_k C V_k$,
where $V_k$ is defined in \eqref{eq:vkdef2} and \eqref{eq:determineVk} 
and $V_*$ is defined in \Cref{sec:propmulti2D}.
In this section, for brevity, we write 
$V_k=[v_{1}, \, v_2]$ and $V_*= [v_{1,*},\, v_{2,*}]$.

\begin{lemma}\label{lemmaRQItypeII}
Let $|\mu_k-\mu_*|\leq\epsilon^2$, $|\lambda_k-\lambda_*|\leq\epsilon^2$, 
and $\dist(x_k,\mathcal{X}_*)\leq\epsilon$ for some $\epsilon\leq1$. 
Then there exist $\epsilon_2>0$ and $\kappa_4>0$ depending on 
$(A,C,\mu_*,\lambda_*)$, such that if $\epsilon\leq\epsilon_2$, 
we have
\begin{equation}\label{eq:thmbehavior2eqs1}
\left\|V_*-V_k\begin{bmatrix}
\gamma_1&\\
&\gamma_2
\end{bmatrix}\right\|\leq \kappa_4\epsilon^2,
\end{equation}
where $\gamma_i=\sign(v_i^Hv_{i,*})$.
\end{lemma}
\begin{proof} 
Consider the nullspace of $\widehat{J}_k$ and 
$\widehat{J}(\mu_*,\lambda_*,x_k)$. 
If $\epsilon\leq\epsilon_1$, where $\epsilon_1$ is defined 
in \Cref{Thm:behaviorII}, we have
\begin{equation}\label{eq:sigmaWidehatJ}
\begin{aligned}
\sigma_{n}(\widehat{J}_k)&\geq\frac{1}{2}\sigma_{n,\widehat{J}_*}>0, \\
\sigma_{n}(\widehat{J}(\mu_*,\lambda_*,x_k))&\geq\frac{1}{2}\sigma_{n,\widehat{J}_*}>0.
\end{aligned}
\end{equation}
Thus their nullspace has dimension 2. 
The nullspace of $\widehat{J}(\mu_*,\lambda_*,x_k)$ is spanned 
by $\widehat{V}_*=\left[\begin{smallmatrix}
V_*\\0\\0
\end{smallmatrix}\right]$. 

Now denote $\widehat{V}_k$ as an orthonormal basis of the nullspace of 
$\widehat{J}_k$. According to \Cref{Thm:null}, 
if $\|\widehat{J}_k-\widehat{J}(\mu_*,\lambda_*,x_k)\|
\leq\frac{1}{2}\sigma_{n}(\widehat{J}(\mu_*,\lambda_*,x_k))$, we have
\begin{equation}\label{eq:lemmaRQItypeIIeqs1}
\|\sin\Theta(\widehat{V}_k,\widehat{V}_*)\|\leq \frac{8\|\widehat{J}(\mu_*,\lambda_*,x_k)\|}{\sigma_{n}^2(\widehat{J}(\mu_*,\lambda_*,x_k))}\|\widehat{J}_k-\widehat{J}(\mu_*,\lambda_*,x_k)\|.
\end{equation}
Straight calculation shows 
\begin{align*}
&\|\widehat{J}_k-\widehat{J}(\mu_*,\lambda_*,x_k)\|\leq (\|C\|+1)\epsilon^2\\
&\|\widehat{J}(\mu_*,\lambda_*,x_k)\|\leq \|A-\mu_*C-\lambda_*I\|+\|C\|+1.
\end{align*}
Therefore, if $\epsilon\leq\widetilde{\epsilon_1}\equiv\sqrt{\frac{\sigma_{n,\widehat{J}_*}}{4(\|C\|+1)}}$, 
\[\|\widehat{J}_k-\widehat{J}(\mu_*,\lambda_*,x_k)\|\leq(\|C\|+1)\epsilon^2\leq\frac{1}{2}\sigma_{n}(\widehat{J}(\mu_*,\lambda_*,x_k)),\]
where we use \eqref{eq:sigmaWidehatJ}. 
Thus the inequality \eqref{eq:lemmaRQItypeIIeqs1} holds. 
Furthermore, we have
\begin{equation}
\|\sin\Theta(\widehat{V}_k,\widehat{V}_*)\|
\leq \frac{32(\|A-\mu_*C-\lambda_*I\|+\|C\|+1)}
{\sigma_{n,\widehat{J}_*}^2}(\|C\|+1)\epsilon^2
\equiv T\epsilon^2.
\end{equation}
According to \Cref{lemma:Zhang}, there exists unitary matrix $\widehat{Z}$, 
such that
\[
\|\widehat{V}_*-\widehat{V}_k\widehat{Z}\|\leq\sqrt{2}\|\sin\Theta(\widehat{V}_k, 
\widehat{V}_*)\|\leq\sqrt{2}T\epsilon^2.
\]
Comparing the first $n$ rows of $\widehat{V}_*-\widehat{V}_k\widehat{Z}$, 
we obtain
\[
\|V_*-\widehat{V}_k(1:n,:)\widehat{Z}\|\leq\sqrt{2}T\epsilon^2.
\]
According to \Cref{Thm:lu}, if $\sqrt{2}T\epsilon^2\leq\frac{1}{2}$, 
i.e., $\epsilon\leq\widetilde{\epsilon}_2\equiv 
\frac{1}{8^{\frac{1}{4}}\sqrt{T}}$, we have 
\[
\|\sin\Theta(V_*, V_k)\| 
= \|\sin\Theta(V_*, \widehat{V}_k(1:n,:)\widehat{Z})\|\leq2\sqrt{2}T\epsilon^2.
\]
	
Note that $V_*^HCV_*=\Diag(c_{1,*},c_{2,*})$, $V_k^HCV_k=\Diag(c_{1,k},c_{2,k})$. 
Thus according to \Cref{Thm:doubleV}, there exist positive constants $t_0,\kappa_0$ depending on $C,c_{1,*},c_{2,*}$, such that if $2\sqrt{2}T\epsilon^2\leq t_0$, i.e., $\epsilon\leq\widetilde{\epsilon}_3\equiv\sqrt{\frac{t_0}{2\sqrt{2}T}}$, for $\gamma_i=\sign(v_i^Hv_{i,*})$, we have 
\begin{equation*}
\|V_*-V_k\Diag(\gamma_1,\gamma_2)\|\leq \kappa_0\|\sin\Theta(V_*,V_k)\|\leq 2\sqrt{2}\kappa_0T\epsilon^2.
\end{equation*}
Let $\kappa_4=2\sqrt{2}\kappa_0T$ and 
$\epsilon_2=\min\{1,\epsilon_1,\widetilde{\epsilon}_1,\widetilde{\epsilon}_2,\widetilde{\epsilon}_3\}$. Then we reach the conclusion. 	
\end{proof}

The next lemma shows the the matrix $C_k$ is nonsingular. 

\begin{lemma}\label{lem:CkindefII}
Let $|\mu_k-\mu_*|\leq\epsilon^2$, $|\lambda_k-\lambda_*|\leq\epsilon^2$, 
and $\dist(x_k,\mathcal{X}_*)\leq\epsilon$ for some $\epsilon\leq1$. 
There exists $\epsilon_3>0$ depending on 
$(A,C,\mu_*,\lambda_*)$, such that if $\epsilon\leq\epsilon_3$, 
then $C_k$ is indefinite and 
\begin{equation}\label{eq:CkindefIIineqs}
c_{1,k}\geq\frac{c_{1,*}}{2}>0
\quad \mbox{and}\quad 
c_{2,k}\leq\frac{c_{2,*}}{2}<0.
\end{equation}
\end{lemma}
\begin{proof}
Let $\epsilon\leq\epsilon_2$, where $\epsilon_2$ is defined in \Cref{lemmaRQItypeII}.
Then \eqref{eq:thmbehavior2eqs1} holds. Note that 
\[
c_{i,*}=v_{i,*}^HCv_{i,*},\quad c_{i,k}=(v_{i}\gamma_i)^HC(v_{i}\gamma_i).
\] 
Therefore with simple analysis we can find $\epsilon_3$ depending on 
$A,C,\mu_*,\lambda_*$, such that if $\epsilon\leq\epsilon_3$, 
the inequality \eqref{eq:CkindefIIineqs} holds.	
\end{proof}

\subsection{Main result}
We now present the main result on
the quadratic convergence rate of the 2DRQI 
to compute a nonsingular multiple 2D eigenvalue.

\begin{theorem}\label{Thm:RQItype2}
Assume $|\mu_k-\mu_*|\leq\epsilon^2$, $|\lambda_k-\lambda_*|\leq\epsilon^2$, 
and $\dist(x_k,\mathcal{X}_*)\leq\epsilon$ for some $\epsilon\leq1$.
There exist positive constants $\epsilon_0,\kappa_1,\kappa_2,\kappa_3$ 
depending on $(A,C,\mu_*,\lambda_*)$, such that
\[
|\mu_{k+1}-\mu_*|\leq \kappa_1\epsilon^4, \quad
|\lambda_{k+1}-\lambda_*|=\kappa_2\epsilon^4 \quad \mbox{and} \quad
\dist(x_{k+1}, \mathcal{X}_*)\leq \kappa_3\epsilon^2.
\] 
\end{theorem}

For brevity, we assume in the below that we have multiplied 
the diagonal scaling matrix $\begin{bmatrix}
\gamma_1&\\
&\gamma_2
\end{bmatrix}$ on the right of $V_k$ defined in \Cref{lemmaRQItypeII}
such that
\[ 
\left\|V_*-V_k\right\|\leq \kappa_4\epsilon^2, 
\]
and 
\begin{equation}\label{eq:vstarvisreal} 
v_{i,*}^Hv_i = |v_{i,*}^Hv_i|, \quad i=1,2, 
\end{equation}
where $\kappa_4$ is defined in \Cref{lemmaRQItypeII}. 

We first prove the convergence of the 2D eigenvector.

\begin{lemma}\label{lemma2RQItype2}
Under the assumption of 
\Cref{Thm:RQItype2}, there exists positive constants 
$\kappa_3,\epsilon_4>0$ depending on $(A,C,\mu_*,\lambda_*)$, 
such that if $\epsilon\leq\epsilon_4$, then 
\begin{enumerate}[(i)]  
\item $x_{k+1}$ will be determined as in $\eqref{eq:update1a}$ 
           or $\eqref{eq:update1b}$.
\item $\dist(x_{k+1}, \mathcal{X}_*)\leq \kappa_3\epsilon^2$.
\end{enumerate} 
\end{lemma}
\begin{proof}
We first assume $\epsilon\leq\min\{\epsilon_1,\epsilon_2,\epsilon_3\}$, 
where $\epsilon_1, \epsilon_2$ and $\epsilon_3$ are defined in 
\Cref{Thm:behaviorII,lemmaRQItypeII,lem:CkindefII}. Then 
$C_k$ is indefinite. Therefore, $x_{k+1}$ is
computed by $\eqref{eq:update1a}$ or $\eqref{eq:update1b}$, 
and satisfies
\begin{equation} \label{eq:xk1exp}
x_{k+1}\in\Span\{V_k\}, \quad 
x_{k+1}^HCx_{k+1}=0,\quad \mbox{and} \quad
x_{k+1}^Hx_{k+1}=1.
\end{equation} 
This reaches the result (i). 

For the result (ii), first note that the equation \eqref{eq:xk1exp}
implies there exist $\gamma_i^{(k+1)}$ with 
$|\gamma_i^{(k+1)}|=1$ for $i = 1, 2$ such that
\begin{equation}\label{eq:definexplusII}
x_{k+1}=\gamma_{1}^{(k+1)}tv_{1}+\gamma_{2}^{(k+1)}sv_{2},
\end{equation}
where 
\[
t = \sqrt{\frac{-v_{2}^HCv_{2}}{v_{1}^HCv_{1}-v_{2}^HCv_{2}}} 
\quad \mbox{and} \quad
s = \sqrt{\frac{v_{1}^HCv_{1}}{v_{1}^HCv_{1}-v_{2}^HCv_{2}}}.
\] 
On the other hand, {by the definition of $\mathcal{X}_*$ in 
\eqref{eq:XstarStructureII}}, the vector 
\[ 
\gamma_1^{(k+1)}t_*v_{1,*}+\gamma_2^{(k+1)}s_*v_{2,*}\in\mathcal{X}_*,
\]
where
\begin{align*}    
t_{*}  = \sqrt{\frac{-c_{2,*}}{c_{1,*}-c_{2,*}}} 
        = \sqrt{\frac{-v_{2,*}^HCv_{2,*}}{v_{1,*}^HCv_{1,*}-v_{2,*}^HCv_{2,*}}},
\quad 
s_*  =\sqrt{\frac{c_{1,*}}{c_{1,*}-c_{2,*}}} 
      = \sqrt{\frac{v_{1,*}^HCv_{1,*}}{v_{1,*}^HCv_{1,*}-v_{2,*}^HCv_{2,*}}}.
\end{align*}
Consequently, 
\begin{align*}
\|x_{k+1}-(t_*v_{1,*}\gamma_1^{(k+1)}+s_*v_{2,*}\gamma_2^{(k+1)})\| 
& =\|(\gamma_1^{(k+1)}tv_{1}+\gamma_2^{(k+1)}sv_{2})-
     (\gamma_1^{(k+1)}t_*v_{1,*}+\gamma_2^{(k+1)}s_*v_{2,*})\| \\
& \leq\|(t-t_*)v_1\|+|t_*|\|v_1-v_{1,*}\|+\|(s-s_*)v_2\|+|s_*|\|v_2-v_{2,*}\|\\
& \leq |t-t_*|+|s-s_*|+2\kappa_4\epsilon^2.
\end{align*}
We next estimate $|t-t_*|$ and $|s-s_*|$. Let $\Delta v_i = v_{i,*}-v_i$, 
we have
\begin{align*}
t = \sqrt{\frac{-v_{2}^HCv_{2}}{v_{1}^HCv_{1}-v_{2}^HCv_{2}}}
=\sqrt{\frac{-c_{2,*}+\Delta v_{2}^HCv_{2}+v_{2,*}^HC\Delta v_{2}}{c_{1,*}-c_{2,*}-\Delta v_{1}^HCv_{1}-v_{1,*}^HC\Delta v_{1}+\Delta v_{2}^HCv_{2}+v_{2,*}^HC\Delta v_{2}}}
\end{align*}
and
\begin{align*} 
s = \sqrt{\frac{v_{1}^HCv_{1}}{v_{1}^HCv_{1}-v_{2}^HCv_{2}}}
=\sqrt{\frac{c_{1,*}-\Delta v_{1}^HCv_{1}-v_{1,*}^HC\Delta v_{1}}{c_{1,*}-c_{2,*}-\Delta v_{1}^HCv_{1}-v_{1,*}^HC\Delta v_{1}+\Delta v_{2}^HCv_{2}+v_{2,*}^HC\Delta v_{2}}}.	
\end{align*}
Utilizing \Cref{le:basiclemma}(i), there exist positive constants 
$\widetilde{\kappa}_t,\widetilde{\kappa}_s,\widetilde{\epsilon}_0>0$ 
depending on $(A,C,\mu_*,\lambda_*,c_{1,*},c_{2,*})$, 
such that when $\epsilon\leq\widetilde{\epsilon}_0$, we have
\[
|t-t_*| \leq \widetilde{\kappa}_t\epsilon^2 
\quad \mbox{and} \quad
|s-s_*| \leq \widetilde{\kappa}_s\epsilon^2.
\]
Therefore
\[
\|x_{k+1}-(t_*v_{1,*}\gamma_1^{(k+1)}+s_*v_{2,*}\gamma_2^{(k+1)})\|\leq (\widetilde{\kappa}_t+\widetilde{\kappa}_s+2\kappa_4)\epsilon^2.
\]
Since $c_{1,*},c_{2,*}$ are uniquely determined by 
$(A,C,\mu_*,\lambda_*)$, 
$\widetilde{\kappa}_t,\widetilde{\kappa}_s,
\widetilde{\epsilon}_0$ only depend on $(A,C,\mu_*,\lambda_*)$. 
Let $\kappa_3=\widetilde{\kappa}_t+\widetilde{\kappa}_s+2\kappa_4$, 
$\epsilon_4=\min\{\epsilon_1,\epsilon_2,\epsilon_3,\widetilde{\epsilon}_0\}$,
then we have the result (ii).  
\end{proof}

In the following, for brevity,  we assume 
$\gamma_1^{(k+1)}=\gamma_2^{(k+1)}=1$ in \eqref{eq:definexplusII}, and thus
\begin{equation}\label{eq:defxkplus1New}
x_{k+1}=tv_{1}+sv_{2}.
\end{equation}
Recall that from the proof of \Cref{lemma2RQItype2},
\begin{equation}\label{eq:deftildexstarNew}
\widetilde{x}_*\equiv t_*v_{1,*}+s_*v_{2,*}
\end{equation}
satisfies
\[
\|x_{k+1}-\widetilde{x}_*\| \leq \kappa_3\epsilon^2. 
\] 
The following lemma gives out the approximation property of 2D Ritz values. 

\begin{lemma}\label{Thm:verylong}
Under the assumption of 
\Cref{Thm:RQItype2}, there exist positive constants 
$\kappa_1,\kappa_2,\epsilon_0$ depending on $(A,C,\mu_*,\lambda_*)$, 
where $\epsilon_0$ is no larger than $\epsilon_3$ 
defined in \Cref{lemma2RQItype2}, such that if 
$\epsilon\leq\epsilon_0$, then
\begin{enumerate}[(i)] 
\item $(\mu_{k+1},\lambda_{k+1},x_{k+1})$ 
will be determined as in \eqref{eq:update1a} or \eqref{eq:update1b}, and 

\item $|\mu_{k+1}-\mu_*|\leq \kappa_1\epsilon^4$ and
     $|\lambda_{k+1}-\lambda_*|\leq \kappa_2\epsilon^4$.
\end{enumerate} 
\end{lemma}
\begin{proof}
Let $\epsilon\leq\epsilon_4$, where $\epsilon_4$ is defined 
in \Cref{lemma2RQItype2}. Then $C_k$ is indefinite and 
$(\mu_{k+1},\lambda_{k+1},x_{k+1})$ will be computed 
by \eqref{eq:update1a} or \eqref{eq:update1b}. This is
the result (i). 

Now we consider the result (ii), 
{by the expressions $(\mu_{k+1},\lambda_{k+1})$ in 
\eqref{eq:update1a} or \eqref{eq:update1b},}
\begin{equation}\label{eq:updatemulamFromx}
\mu_{k+1}=\frac{\Real(x_{k+1}^HCV_{k}V_k^HAx_{k+1})}{\|V_k^HCx_{k+1}\|^2}, 
\quad 
\lambda_{k+1}=x_{k+1}^HAx_{k+1}.
\end{equation}
Note that $x_{k+1}^HCx_{k+1}=0$ and $x_{k+1}^Hx_{k+1}=1$. 
By the second-order estimate in \Cref{Thm:lsq}, we obtain
the desired error bound for $\lambda_{k+1}$: 
\begin{equation} \label{eq:lambdabd2} 
|\lambda_{k+1}-\lambda_*|\leq \|A-\mu_*C-\lambda_*I\|\kappa_3^2\epsilon^4.
\end{equation} 
For the rest, we only need to estimate the bound of 
the approximation $|\mu_{k+1}-\mu_*|$. 
Let start with an alternative expression of $\mu_*$: 
\begin{align*}
\frac{\widetilde{x}_*^{H}CV_*V_*^HA\widetilde{x}_*}{\|V_*^{H}C\widetilde{x}_*\|^2}
& = \frac{\widetilde{x}_*^{H}CV_*V_*^H(\mu_*C\widetilde{x}_*+\lambda_*\widetilde{x}_*)}{\|V_*^{H}C\widetilde{x}_*\|^2} \\
& = \frac{\mu_*\widetilde{x}_*^{H}CV_*V_*^HC\widetilde{x}_*}{\|V_*^{H}C\widetilde{x}_*\|^2} + \frac{\lambda_*\widetilde{x}_*^{H}CV_*V_*^H\widetilde{x}_*}{\|V_*^{H}C\widetilde{x}_*\|^2} \\ 
& = \mu_* + \frac{\lambda_*\widetilde{x}_*^{H}C\widetilde{x}_*}{\|V_*^{H}C\widetilde{x}_*\|^2} = \mu_*.
\end{align*} 
Next let $\Delta x_{k+1} = \widetilde{x}_*-x_{k+1}$, 
$\Delta V_k = V_*-V_k$, where $\widetilde{x}_*$ is defined in 
\eqref{eq:deftildexstarNew}. Then 
\[ 
\|\Delta x_{k+1}\|\leq \kappa_3\epsilon^2
\quad \mbox{and} \quad 
\|\Delta V_k\|\leq \kappa_4\epsilon^2,
\]
where $\kappa_3,\kappa_4$ are defined 
in \Cref{lemma2RQItype2,lemmaRQItypeII}, respectively.  
%
%
Using \Cref{le:basiclemma}, 
by the definition \eqref{eq:updatemulamFromx} of $\mu_{k+1}$,
\begin{align}
\mu_{k+1} & =\frac{\Real(x_{k+1}^HCV_{k}V_k^HAx_{k+1})}{\|V_k^HCx_{k+1}\|^2} 
      \nonumber \\
& = \frac{\Real((\widetilde{x}_*-\Delta x_{k+1})^HC(V_*-\Delta V_{k})(V_*-\Delta V_{k})^HA(\widetilde{x}_*-\Delta x_{k+1}))}{\|(V_*-\Delta V_{k})^HC(\widetilde{x}_*-\Delta x_{k+1})\|^2} \nonumber \\
& =\frac{\Real((\widetilde{x}_*-\Delta x_{k+1})^HC(V_*-\Delta V_{k})(V_*-\Delta V_{k})^HA(\widetilde{x}_*-\Delta x_{k+1}))}{(\widetilde{x}_*^{H}CV_*-\widetilde{x}_*^{H}C\Delta V_k-\Delta x_{k+1}^HCV_*)(V_*^{H}C\widetilde{x}_*-\Delta V_k^HC\widetilde{x}_*-V_*^{H}C\Delta x_{k+1})+O(\epsilon^4)} \nonumber \\
& = \frac{\widetilde{x}_*^HCV_*V_*^HA\widetilde{x}_* - \Real(\Delta x_{k+1}^HCV_*V_*^HA\widetilde{x}_*+\widetilde{x}_*^HC\Delta V_{k}V_*^HA\widetilde{x}_*+\widetilde{x}_*^HCV_*\Delta V_{k}^HA\widetilde{x}_*)}{\|V_*^HC\widetilde{x}_*\|^2-2\Real(\widetilde{x}_*^HCV_*\Delta V_k^HC\widetilde{x}_*)-2\Real(\widetilde{x}_*^HCV_*V_*^HC\Delta x_{k+1})+O(\epsilon^4)} \nonumber \\
& \quad -\frac{\Real(\widetilde{x}_*^HCV_*V_*^HA\Delta x_{k+1})+O(\epsilon^4)}{\|V_*^HC\widetilde{x}_*\|^2-2\Real(\widetilde{x}_*^HCV_*\Delta V_k^HC\widetilde{x}_*)-2\Real(\widetilde{x}_*^HCV_*V_*^HC\Delta x_{k+1})+O(\epsilon^4)} \nonumber \\
& = \mu_*-\frac{\Real(\Delta x_{k+1}^HCV_*V_*^HA\widetilde{x}_*+\widetilde{x}_*^HC\Delta V_kV_*^HA\widetilde{x}_*+\widetilde{x}_*^HCV_*\Delta V_k^HA\widetilde{x}_*+\widetilde{x}_*^HCV_*V_*^HA\Delta x_{k+1})}{\|V_*^HC\widetilde{x}_*\|^2} \nonumber \\
& \quad +\frac{2\mu_*\Real(\widetilde{x}_*^HCV_*\Delta V_k^HC\widetilde{x}_*+\widetilde{x}_*^HCV_* V_*^HC\Delta x_{k+1})}{\|V_*^HC\widetilde{x}_*\|^2} +O(\epsilon^4), 
\label{eq:typeIIeqs1}
\end{align}
where in the third and fourth equalities, constants in $O(\epsilon^4)$ 
only depend on $(\|A\|$,$\|C\|$,$\kappa_3$,$\kappa_4)$;
in the fifth equality, we use \Cref{le:basiclemma}(ii) and note that
\[
\widetilde{x}_*^HCV_*V_*^HA\widetilde{x}_*
=\widetilde{x}_*^HCV_*V_*^H(\lambda_*\widetilde{x}_*+\mu_*C\widetilde{x}_*)
=\mu_*\|V_*^HC\widetilde{x}_*\|^2.
\]
The corresponding constants in $O(\epsilon^4)$ depend on 
$(\|A\|,\|C\|,\kappa_3,\kappa_4,\sigma_{1,VC_*})$, 
where $\sigma_{1,VC_*}$ is defined in \Cref{lem:cstarII}. 
$\kappa_3,\kappa_4,\sigma_{1,VC_*}$ only depend on $(A,C,\mu_*,\lambda_*)$. 
Note that $A\widetilde{x}_*=\mu_*C\widetilde{x}_*+\lambda_*\widetilde{x}_*$ 
and $V_*^HA = V_*^H(\mu_*C+\lambda_*I)$. 

After a simple reorganization of {\eqref{eq:typeIIeqs1}}, we have
	\begin{equation}
		\begin{aligned}
			&(\mu_{k+1}-\mu_*)\|V_*^HC\widetilde{x}_*\|^2\\
			=& -\lambda_*\Real\left(\Delta x_{k+1}^HCV_*V_*^{H}\widetilde{x}_*+\widetilde{x}_*^{H}C\Delta V_kV_*^{H}\widetilde{x}_*+\widetilde{x}_*^{H}CV_*\Delta V_{k}^H\widetilde{x}_*\right)\\
			&-\mu_*\Real\left(-\Delta x_{k+1}^HCV_*V_*^{H}C\widetilde{x}_*+\widetilde{x}_*^{H}C\Delta V_kV_*^{H}C\widetilde{x}_*-\widetilde{x}_*^{H}CV_*\Delta V_k^H C\widetilde{x}_*\right)\\
			&-\Real\left(\widetilde{x}_*^{H}CV_*(\mu_*V_*^{H}C+\lambda_*V_*^{H})\Delta x_{k+1}\right)+O(\epsilon^4)\\
			=&O(\epsilon^4)-\lambda_*\Real\left(\Delta x_{k+1}^HC\widetilde{x}_*+\widetilde{x}_*^{H}C\Delta V_kV_*^{H}\widetilde{x}_*+\widetilde{x}_*^{H}CV_*\Delta V_{k}^H\widetilde{x}_*+\widetilde{x}_*^{H}CV_*V_*^{H}\Delta x_{k+1}\right)\\
			&-\mu_*\Real\left(-\Delta x_{k+1}^HCV_*V_*^{H}C\widetilde{x}_*+\widetilde{x}_*^{H}C\Delta V_kV_*^{H}C\widetilde{x}_*-\widetilde{x}_*^{H}CV_*\Delta V_k^H C\widetilde{x}_*+\widetilde{x}_*^{H}CV_*V_*^{H}C\Delta x_{k+1}\right)\\
			=&-\lambda_*\Real(T_{11}+T_{12})-\mu_*\Real(T_{21}+T_{22})+O(\epsilon^4),\\
		\end{aligned}
	\end{equation}
    where
	\[
	\begin{aligned}
		& T_{11}=\Delta x_{k+1}^HC\widetilde{x}_*,\\
		& T_{12}=\widetilde{x}_*^{H}C\Delta V_kV_*^{H}\widetilde{x}_*+\widetilde{x}_*^{H}CV_*\Delta V_{k}^H\widetilde{x}_*+\widetilde{x}_*^{H}CV_*V_*^{H}\Delta x_{k+1}, \\
		& T_{21} = -\Delta x_{k+1}^HCV_*V_*^{H}C\widetilde{x}_*+\widetilde{x}_*^{H}CV_*V_*^{H}C\Delta x_{k+1}, \\
		& T_{22} = \widetilde{x}_*^{H}C\Delta V_kV_*^{H}C\widetilde{x}_*-\widetilde{x}_*^{H}CV\Delta V_k^H C\widetilde{x}_*.
	\end{aligned}
	\]
Note that $T_{21}$ and $T_{22}$ are pure imaginary scalars. Therefore, 
\begin{equation}
\Real(T_{21}+T_{22})=0. 
\end{equation}
Meanwhile, 
\[
0  = x_{k+1}^HCx_{k+1}
 = (\widetilde{x}_*-\Delta x_{k+1})^HC(\widetilde{x}_*-\Delta x_{k+1})
 = -2\Real(\Delta x_{k+1}^HC\widetilde{x}_*)+\Delta x_{k+1}^HC\Delta x_{k+1}.
\]
Therefore, 
\[
|\Real(T_{11})| = |\Real(\Delta x_{k+1}^HC\widetilde{x}_*)|
\leq \frac{\|C\|}{2}\kappa_3^2\epsilon^4, 
\]
and we have
\begin{equation}\label{eq:typeIICombine1}
\mu_{k+1}-\mu_*
= -\lambda_*\frac{\Real(T_{12})}{\|V_*^HC\widetilde{x}_*\|^2}+O(\epsilon^4),
\end{equation}
where the constants in $O(\epsilon^4)$ terms only depend on 
$(A,C,\mu_*,\lambda_*)$. Consequently, we only need to prove that 
\[
\Real(T_{12}) = O(\epsilon^4). 
\] 
We begin by decomposing $v_i$:
\begin{align*}
v_{1} & = E_{11}v_{1,*}+E_{12}v_{2,*}+r_1,\\
v_{2}  &= E_{21}v_{1,*}+E_{22}v_{2,*}+r_2,
\end{align*}
where $r_i\perp v_{j,*}$, $i,j=1,2$.
We first derive estimates on $E_{ij}$. 
According to \eqref{eq:vstarvisreal}, 
$E_{11} = v_{1,*}^Hv_1$ and $E_{22}=v_{2,*}^Hv_2$ are nonnegative real scalars. 
Since $\|V_k-V_*\|\leq \kappa_4\epsilon^2$, we have
\begin{equation}\label{eq1}
|E_{11}-1|^2 + |E_{12}|^2 + \|r_1\|^2 \leq \kappa_4^2\epsilon^4 
\quad \mbox{and} \quad
|E_{21}|^2 + |E_{22}-1|^2 + \|r_2\|^2 \leq \kappa_4^2\epsilon^4.
\end{equation}
Therefore,
\begin{equation}\label{eq:Eeqs1}
\begin{aligned}
&|E_{11}-1|\leq \kappa_4\epsilon^2, \quad |E_{12}|\leq \kappa_4\epsilon^2,\quad \|r_1\|\leq \kappa_4\epsilon^2,\quad |E_{12}|^2+\|r_1\|^2\leq\kappa_4^2\epsilon^4,\\
&|E_{22}-1|\leq \kappa_4\epsilon^2, \quad |E_{21}|\leq \kappa_4\epsilon^2,\quad \|r_2\|\leq \kappa_4\epsilon^2,\quad |E_{21}|^2+\|r_2\|^2\leq\kappa_4^2\epsilon^4.
\end{aligned}
\end{equation}
Note that $|E_{11}|^2+E_{12}^2+\|r_1\|^2=1$. Since $E_{11}$, 
$E_{22}$ are nonnegative real scalars, we have
\[
(E_{11}-1)(E_{11}+1) = E_{11}^2 - 1 = -E_{12}^2-\|r_1\|^2.
\]
Similarly, in terms of $E_{12}, E_{22}$ and $r_2$, we have 
\[
(E_{22}-1)(E_{22}+1) = E_{22}^2 - 1 = -E_{12}^2-\|r_2\|^2.
\]
Combined with \eqref{eq:Eeqs1}, this implies
\begin{equation}\label{eq:Eeqs2}
|E_{11}-1|\leq \kappa_4^2\epsilon^4, \quad
|E_{22}-1|\leq \kappa_4^2\epsilon^4
\end{equation}
Utilizing $v_{1}^Hv_{2}=0$, we have
\[E_{11}E_{21}+\overline{E}_{12}E_{22}+r_1^Hr_2=0,\]
which implies
\begin{align}
|E_{21}+\overline{E}_{12}| 
&\leq |(1-E_{11})E_{21}+\overline{E}_{12}(1-E_{22})|+|E_{11}E_{21}
        +\overline{E}_{12}E_{22}| \nonumber \\ 
&\leq \kappa_4^2\epsilon^4(|E_{21}|+|E_{12}|)+|r_1^Hr_2| \nonumber \\
&\leq 2\kappa_4^3\epsilon^6+\kappa_4^2\epsilon^4.  \label{eq:Eeqs3}
\end{align}
With these {caclulations}, we can give a more detailed 
description of $\Delta x_{k+1}$. 
Recall \eqref{eq:defxkplus1New} and \eqref{eq:deftildexstarNew}: 
\begin{equation}
x_{k+1} = tv_{1}+sv_{2} 
\quad \mbox{and} \quad
\widetilde{x}_* = t_*v_{1,*}+s_*v_{2,*}.
\end{equation}
We will perform further analysis between $t$ and $t_*$, $s$ and $s_*$. 
To simplify the expressions, we introduce the following notations 
\[
\begin{array}{ll}
\rho_{11}=v_{1,*}^HCr_1, & \rho_{12}=v_{1,*}^HCr_2, \\
\rho_{21}=v_{2,*}^HCr_1, & \rho_{22}=v_{2,*}^HCr_2. \\
\end{array}
\]
Utilizing \Cref{le:basiclemma} and equations \eqref{eq:Eeqs1} and \eqref{eq:Eeqs2}, 
we have
\begin{align}
t &= \sqrt{\frac{-v_{2}^HCv_{2}}{v_{1}^HCv_{1}-v_{2}^HCv_{2}}} \nonumber \\	
			&=\sqrt{\frac{-v_{2,*}^HCv_{2,*}-2\Real(v_{2,*}^HC(r_2+E_{21}v_{1,*}))+O(\epsilon^4)}{v_{1,*}^HCv_{1,*}-v_{2,*}^HCv_{2,*}+2\Real(v_{1,*}^HC(r_1+E_{12}v_{2,*}))-2\Real(v_{2,*}^HC(r_2+E_{21}v_{1,*}))+O(\epsilon^4)} } \nonumber \\
& = \sqrt{\frac{-c_{2,*}-2\Real(\rho_{22})+O(\epsilon^4)}{c_{1,*}-c_{2,*}+2\Real(\rho_{11})-2\Real(\rho_{22})+O(\epsilon^4)} } \nonumber \\
& = \sqrt{\frac{-c_{2,*}}{c_{1,*}-c_{2,*}}+\frac{-2\Real(\rho_{22})}{c_{1,*}-c_{2,*}}-\frac{-2c_{2,*}\Real(\rho_{11}-\rho_{22})}{(c_{1,*}-c_{2,*})^2}+O(\epsilon^4)} \nonumber \\
& = \sqrt{t_*^2+\frac{-2\Real(\rho_{22})c_{1,*}+2c_{2,*}\Real(\rho_{11})}
				{(c_{1,*}-c_{2,*})^2}+O(\epsilon^4)} \nonumber \\
& = t_* + \frac{-\Real(\rho_{22})c_{1,*}
+c_{2,*}\Real(\rho_{11})}{t_*(c_{1,*}-c_{2,*})^2}+O(\epsilon^4) \label{eq:Eeqs4}
\end{align}
where the constants in $O(\epsilon^4)$ of the second and third equalities 
only depend on $\kappa_4,\|C\|$; and we use \Cref{le:basiclemma}(ii) 
in the fourth equality; in the last equality, we use \Cref{le:basiclemma}(iii), 
with the corresponding constants in $O(\epsilon^4)$ only depending on 
$(\kappa_4,\|C\|,c_{1,*},c_{2,*})$. These constans can be in turn 
viewed as only depending on $(A,C,\mu_*,\lambda_*)$. 
{
\allowdisplaybreaks
By a similar calculation, we have
\begin{equation}\label{eq:Eeqs5}
s =s_*+
\frac{-\Real(\rho_{11})c_{2,*}+c_{1,*}\Real(\rho_{22})}{s_*(c_{1,*}-c_{2,*})^2}
+O(\epsilon^4).
\end{equation}
Now let us back to the term $T_{12}$. Denote $\Delta v_i = v_{i,*}-v_i$, $i=1,2$. 
Using \eqref{eq:Eeqs1}, \eqref{eq:Eeqs2} and \eqref{eq:Eeqs3},  we have
\begin{align*}
T_{12} &= \widetilde{x}_*^{H}C\left[\begin{bmatrix}
				\Delta v_1 & \Delta v_2
			\end{bmatrix}\begin{bmatrix}
				v_{1,*}^H \\
				v_{2,*}^H
			\end{bmatrix}(t_*v_{1,*}+s_*v_{2,*})+\begin{bmatrix}
				v_{1,*} & v_{2,*}
			\end{bmatrix}\begin{bmatrix}
				\Delta v_1^H \\
				\Delta v_{2}^H
			\end{bmatrix}(t_*v_{1,*}+s_*v_{2,*})\right]\nonumber \\
			& +\widetilde{x}_*^{H}C\left[\begin{bmatrix}
				v_{1,*} & v_{2,*}
			\end{bmatrix}\begin{bmatrix}
				v_{1,*}^H \\
				v_{2,*}^H
\end{bmatrix}(t_*v_{1,*}+s_*v_{2,*}-t v_{1}-s v_{2})\right] \nonumber\\
&= \widetilde{x}_*^{H}C\left[t_*\Delta v_1+s_*\Delta v_2+\begin{bmatrix}
				v_{1,*} & v_{2,*}
			\end{bmatrix}\begin{bmatrix}
				-s_*\overline{E}_{12} \\
				-t_*\overline{E}_{21}
			\end{bmatrix}+\begin{bmatrix}
				v_{1,*} & v_{2,*}
			\end{bmatrix}\begin{bmatrix}
				t_*-t-s E_{21} \\ s_*-s-t E_{12}
			\end{bmatrix}\right]+O(\epsilon^4) \nonumber \\
& = \widetilde{x}_*^{H}C\left[t_*\Delta v_1+s_*\Delta v_2+\begin{bmatrix}
				v_{1,*} & v_{2,*}
			\end{bmatrix}\begin{bmatrix}
	t_*-t-s_*\overline{E}_{12}-s E_{21} \\ s_*-s-t_*\overline{E}_{21}-t E_{12}
	\end{bmatrix}\right]+O(\epsilon^4)\nonumber\\
& = \widetilde{x}_*^{H}C\left[t_*\Delta v_1+s_*\Delta v_2+\begin{bmatrix}
				v_{1,*} & v_{2,*}
			\end{bmatrix}\begin{bmatrix}
				t_*-t+s_*E_{21}-s E_{21} \\ s_*-s+t_*E_{12}-t E_{12}
\end{bmatrix}\right]+O(\epsilon^4)\nonumber\\
& = \widetilde{x}_*^{H}C\left[t_*\Delta v_1+s_*\Delta v_2+v_{1,*}(t_*-t)+v_{2,*}(s_*-s)\right]+O(\epsilon^4)\nonumber\\
& = \widetilde{x}_*^{H}C[-t_*r_1-s_*r_2-t_*E_{12}v_{2,*}-s_*E_{21}v_{1,*}+v_{1,*}(t_*-t)+v_{2,*}(s_*-s)]+O(\epsilon^4).\nonumber
\end{align*}
where all constants in $O(\epsilon^4)$ terms only depend on 
$A,C,\mu_*,\lambda_*$. 
}

According to \eqref{eq:Eeqs4} and \eqref{eq:Eeqs5}, we have
\begin{align*}
T_{12} &= O(\epsilon^4)+\begin{bmatrix}
				t_{*} & s_{*}
			\end{bmatrix}V_*^HC\left[-t_*r_1-s_*r_2+V_*\begin{bmatrix}
				t_*-t-s_*E_{21}	 \\ s_*-s-t_*E_{12}
			\end{bmatrix}\right] \nonumber \\
& = O(\epsilon^4)
  -t_*^2v_{1,*}^HCr_1-t_*s_*v_{1,*}^HCr_2-s_*v_{2,*}^HCt_*r_1-s_*^2v_{2,*}^HCr_2 
      \nonumber \\
&+t_*(t_*-t-s_*E_{21})c_{1,*}+s_*(s_*-s-t_*E_{12})c_{2,*} \nonumber \\
&= O(\epsilon^4)
-t_*^2\rho_{11}-t_*s_*\rho_{12}-t_*s_*\rho_{21}-s_*^2\rho_{22}
-t_*s_*E_{21}c_{1,*}-t_*s_*E_{12}c_{2,*} \nonumber \\
&+c_{1,*}\frac{\Real(\rho_{22})c_{1,*}-c_{2,*}\Real(\rho_{11})}{(c_{1,*}-c_{2,*})^2}+c_{2,*}\frac{\Real(\rho_{11})c_{2,*}-c_{1,*}\Real(\rho_{22})}{(c_{1,*}-c_{2,*})^2}.
\end{align*}
Therefore, 
\begin{align}
\Real(T_{12}) 
&= \Real\left(\frac{c_{2,*}\rho_{11}}{c_{1,*}-c_{2,*}}+c_{1,*}\frac{\rho_{22}c_{1,*}-c_{2,*}\rho_{11}}{(c_{1,*}-c_{2,*})^2}-\frac{c_{1,*}\rho_{22}}{c_{1,*}-c_{2,*}}+c_{2,*}\frac{\rho_{11}c_{2,*}-c_{1,*}\rho_{22}}{(c_{1,*}-c_{2,*})^2}\right)\nonumber\\
& \quad 
-\Real\left(t_*s_*\rho_{12}+t_*s_*\rho_{21}+t_*s_*E_{21}c_{1,*} + t_*s_*E_{12}c_{2,*}\right)+O(\epsilon^4)\nonumber\\
&=-t_*s_*\Real(\rho_{12}+\rho_{21}+E_{21}c_{1,*}+E_{12}c_{2,*})+O(\epsilon^4).\label{eq:typeIICombine2}
\end{align}
Further note that
\begin{align*}
0 &= v_{1}^HCv_{2}
=(v_{1,*}+E_{12}v_{2,*}+r_1)^HC(v_{2,*}+E_{21}v_{1,*}+r_2)+O(\epsilon^4)\\
&= E_{21}c_{1,*}+\overline{E}_{12}c_{2,*}+\rho_{12}+\overline{\rho}_{21}
   +O(\epsilon^4).
\end{align*}
Thus we obtain 
\begin{equation}\label{eq:typeIICombine3}
\Real(\rho_{12}+\rho_{21}+E_{21}c_{1,*}+E_{12}c_{2,*})=O(\epsilon^4)
\end{equation}
and 
\[
\Real(T_{12})=O(\epsilon^4).
\]
Combine with \eqref{eq:typeIICombine1}, \eqref{eq:typeIICombine2}
and \eqref{eq:typeIICombine3}, we 
find positive constants $\widetilde{\epsilon}_0,\kappa_1$ depending 
on $(A,C,\mu_*,\lambda_*)$, such that if $\epsilon\leq\widetilde{\epsilon}_0$, 
we have
\begin{equation} \label{eq:mubd2} 
|\mu_{k+1}-\mu_*|\leq \kappa_1\epsilon^4.
\end{equation} 
Combine the bounds \eqref{eq:lambdabd2} and \eqref{eq:mubd2},  
we reach the result (ii) by letting
$\epsilon_0=\min\{\epsilon_4,\widetilde{\epsilon}_0\}$.
\end{proof}

\paragraph{Proof of Theorem \ref{Thm:RQItype2}.}  
The theorem is a direct consequence of 
\Cref{lemma2RQItype2,Thm:verylong}. 
\hfill $\Box$.


\section{Conclusions}\label{sec:conclusion}
In this part, we presented 
a rigorous convergence analysis of the proposed 2DRQI,
and showed that the 2DRQI is locally quadratically convergent
when the target 2D-eigentriplet is nonsingular. 
The quadratical convergence of the 2DRQI is verified 
by numerical examples presented in Part II of this work \cite[Sec.6]{2DEVPII}. 


\appendix

\section{Proof of Theorem~2.1 in Part II} \label{subsec:proofThm25}

The following theorem shows that if a 2D-eigenvector is near
the subspace spanned by an $n\times 2$ orthonormal matrix $V$, 
then the 2D Ritz triplets of the 2D Rayleigh quotient $(V^H A V, V^H C V)$
induced by $V$ will contain a good approximation to a 2D-eigentriplet. 

\begin{theorem}\label{Thm2DRQI}
Let $(\mu_*,\lambda_*,x_*)$ be a 2D-eigentriplet of $(A,C)$.
For $\gamma>0$, denote $\mathcal{V}_{\gamma}$ as the 
set of $n\times2$ orthonormal matrices $V$ satisfying
\begin{enumerate}
\item $V^HCV$ is indefinite and diagonal,
\item $|(V^HAV)_{12}|\geq\gamma$,
\item $\left|\det(V^HCV)\right|\geq \gamma$.
\end{enumerate}
Then there exist positive constants $\alpha_1$, $\alpha_2$ and $\alpha_3$ 
only depending on $(A, C, \mu_*, \lambda_*)$ and $\gamma$,
such that for any $V\in\mathcal{V}_{\gamma}$, let 
\[ 
\epsilon=\dist(x_*, \Span\{V\})\equiv\min\{\|x_*-v\|\ |\ v\in\Span\{V\}\} 
\] 
and assume $\epsilon<1$, there exists a 2D Ritz triplet 
$(\nu,\theta, Vz)$ satisfying
\[
|\nu-\mu_*| \leq \alpha_1 \epsilon,  \quad
|\theta -\lambda_*| \leq \alpha_2 \epsilon^2
\quad \mbox{and} \quad
\|Vz-x_*\| \leq \alpha_3 \epsilon.
\]	
\end{theorem}
\begin{proof} 
The proof is similar to the proof of the quadratic convergence of
the 2DRQI for computing a nonsingular simple 2D eigentriplet 
presented in \Cref{sec:converge_analysis_II}. 

Let
\[
V^HAV = \begin{bmatrix}
a_{11} & a_{12}\\
a_{21} & a_{22}
\end{bmatrix},  \quad 
V^HCV = \Diag(c_1,c_2) \,\, \mbox{with}\,\, c_1 > 0 > c_2
\]
and
\[ 
\mathcal{X} = \{x\ |\ x\in \Span\{V\}, x^HCx =0, x^Hx = 1\}.
\]
Assume $\widehat{x}^{\rm (p)}\in\Span\{V\}$ satisfies 
$\|x_*-\widehat{x}^{\rm (p)}\|=\epsilon$. Then 
since $\widehat{x}^{\rm (p)}$ is the closest vector to $x_*$ in $\Span\{V\}$, $\widehat{x}^{{\rm (p)}H}(x_*-\widehat{x}^{\rm (p)}) = 0$ and thus, 
\[
1 = \|x_*\| = \sqrt{\|\widehat{x}^{\rm (p)}\|^2+\|x_*-\widehat{x}^{\rm (p)}\|^2}
=\sqrt{\|\widehat{x}^{\rm (p)}\|^2+\epsilon^2}.
\]
Since $\epsilon<1$, we have $\widehat{x}^{\rm (p)}\neq 0$.
Define $\widehat{x} = \widehat{x}^{\rm (p)}/\|\widehat{x}^{\rm (p)}\|$. 
Then $\|\widehat{x}\|=1$ and
\begin{align*}
\|\widehat{x}-x_*\| 
&= \sqrt{2-2\Real(x_*^H\widehat{x}^{\rm (p)})/\|\widehat{x}^{\rm (p)}\|}\\
&= \sqrt{2-2\|\widehat{x}^{\rm (p)}\|}
=\sqrt{2-2\sqrt{1-\epsilon^2}}
\leq \sqrt{2}\epsilon.	
\end{align*}
In the following, we denote
$\mathcal{X}_{x_*}=\{\gamma x_* \mid \gamma \in \mathbb{C}, \|\gamma|=1\}$.

Recall that in \Cref{prop2}, we proved that 
(c.f., \eqref{eq:xcxk},\eqref{eq:xxpdiff},\eqref{eq:lemmaprop2eqs1},
\eqref{eq:lemmaprop2eqs2}), 
if $V_k^HCV_k$ is indefinite, then for any unit vector $x \in \Span\{V_k\}$, 
we can find 
\[
\widetilde{x}\in\mathcal{X}_k\equiv\{y \mid y\in\Span\{V_k\}, y^Hy=1, y^HCy=0\},
\]
such that
\begin{equation}
\dist(\widetilde{x},\mathcal{X}_{x_*})\leq \left(\frac{2\|C\|}{\sqrt{-c_{1,k}c_{2,k}}}+1\right)\dist(x,\mathcal{X}_{x_*}), 
\end{equation}
where $c_{1,k},c_{2,k}$ are eigenvalues of $V_k^HCV_k$. 
		
Since under the assumption, $V^HCV$ is indefinite. Therefore, 
after substituting $V$, $\mathcal{X}$ for $V_k$, $\mathcal{X}_k$, respectively,
we can also prove that for the unit vector $\widehat{x}$ in $\Span\{V\}$, 
there exists $\widetilde{x}\in\mathcal{X}$, such that 
\begin{equation}\label{eq:2DRQ1}
\dist(\widetilde{x},\mathcal{X}_{x_*})\leq \left(\frac{2\|C\|}{\sqrt{-c_{1}c_{2}}}+1\right)\dist(\widehat{x}, \mathcal{X}_{x_*})\leq \left(\frac{2\sqrt{2}\|C\|}{\sqrt{-c_{1}c_{2}}}+\sqrt{2}\right)\epsilon.
\end{equation}
Now we prove the approximation property of 2D Ritz teiplets.
In the proof of \Cref{lemma:xkp}, we only use the following conditions.
\begin{itemize}
\item When $\epsilon\leq\min\{\epsilon_1,\epsilon_2,\epsilon_T\}$, $\widetilde{x}_{k+1}^{\rm (p)}$ is well defined and thus we can find $\widetilde{x}_{k+1}\in\mathcal{X}_k$ such that $\dist(\widetilde{x}_{k+1},\mathcal{X}_{x_*})$ is sufficiently small, i.e., \eqref{eq:lemmaprop2eqs2}. 

\item When $\epsilon\leq\epsilon_2$, $|a_{12,k}|\geq \frac{|a_{12,*}|}{2}>0$. 

\item When $\epsilon\leq\epsilon_T$, $c_{1,k}>0,c_{2,k}<0,\sqrt{-c_{1,k}c_{2,k}}\geq\frac{\sqrt{-c_{1,*}c_{2,*}|}}{2}>0$. 
\end{itemize}
Under these conditions, we prove that there exists a 2D Ritz triplet $(\nu_{k,j},\theta_{k,j},x_{k,j})$ such that (c.f., \eqref{eq:sinTildeTheta}\eqref{eq:lemmaxkpeqs3}\eqref{eq:muproject}\eqref{eq:lemmaxkpseqs5}\eqref{eq:lemmaxkpseqs6}\eqref{eq:lamboundstype1}): 
\begin{equation}\label{eq:2DRQ2}
\begin{aligned}
				\dist(x_{k,j},\mathcal{X}_{x_*})&\leq \left(\sqrt{2}\frac{\|A-\mu_*C-\lambda_*I\|}{|a_{12,k}|}+1\right)\dist(\widetilde{x},\mathcal{X}_{x_*})\\
				|\nu_{k,j}-\mu_*|&\leq \frac{\|A-\mu_*C-\lambda_*I\|}{\sqrt{-c_{1,k}c_{2,k}}}\dist(x_{k,j},\mathcal{X}_{x_*}),\\
				|\theta_{k,j}-\lambda_*|&\leq \|A-\mu_*C-\lambda_*I\|\dist(x_{k,j},\mathcal{X}_{x_*})^2.
\end{aligned}
\end{equation}
Note that we have now found $\widetilde{x}\in\mathcal{X}$ such that 
\[
\dist(\widetilde{x},\mathcal{X}_{x_*})\leq 
\left(\frac{2\sqrt{2}\|C\|}{\sqrt{-c_{1}c_{2}}}+\sqrt{2}\right)\epsilon. 
\] 
Furthermore, we have $|(V^HAV)_{12}|\geq\gamma>0$ and  
$-c_{1}c_{2}=-\det(V^HCV)\geq \gamma>0$. 
Therefore, we can follow the same argument without adding requirements to $\epsilon$, and finally prove that $(V^HAV,V^HCV)$ has two 2D Ritz triplets with one $(\nu,\theta,V\widehat{z})$ of them satisfies:
\begin{equation}\label{eq:2DRQ3}
\begin{aligned}
\dist(V\widehat{z},\mathcal{X}_{x_*})&\leq \left(\sqrt{2}\frac{\|A-\mu_*C-\lambda_*I\|}{|a_{12}|}+1\right)\dist(\widetilde{x},\mathcal{X}_{x_*})\\
|\nu-\mu_*|&\leq \frac{\|A-\mu_*C-\lambda_*I\|}{\sqrt{-c_{1}c_{2}}}\dist(V\widehat{z},\mathcal{X}_{x_*}),\\
|\theta-\lambda_*|&\leq 		\|A-\mu_*C-\lambda_*I\|\dist(V\widehat{z},\mathcal{X}_{x_*})^2.
\end{aligned}
\end{equation}
Using \eqref{eq:2DRQ1} and the assumptions, we have
\begin{equation}\label{eq:2DRQ3}
\begin{aligned}
\dist(V\widehat{z},\mathcal{X}_{x_*})&\leq \left(\sqrt{2}\frac{\|A-\mu_*C-\lambda_*I\|}{\gamma}+1\right)\left(\frac{2\sqrt{2}\|C\|}{\sqrt{\gamma}}+\sqrt{2}\right)\epsilon\equiv\alpha_1\epsilon\\
|\nu-\mu_*|&\leq \frac{\|A-\mu_*C-\lambda_*I\|}{\sqrt{\gamma}}\alpha_1\epsilon\equiv\alpha_2\epsilon,\\
|\theta-\lambda_*|&\leq 		\|A-\mu_*C-\lambda_*I\|\alpha_1^2\epsilon^2\equiv\alpha_3\epsilon^2.
\end{aligned}
\end{equation}
Let $\gamma$ satisfy $|\gamma|=1$ and 
$\|V\widehat{z}-\gamma x_*\|=\dist(V\widehat{z},\mathcal{X}_{x_*})$. 
Then $\|\overline{\gamma}V\widehat{z}- x_*\|\leq\alpha_1\epsilon$. 
Let $z = \overline{\gamma}\widehat{z}$,  then we reach the conclusion.
\end{proof}


\section{Basic inequalities}
We present three basic inequalities that are used in the convergence analysis.

\begin{lemma}\label{le:basiclemma}
Let $g_1,g_2,r_1,r_2,\epsilon$ be functions of $x\in\mathbb{C}^n$, 
where $\epsilon$ can be viewed as a parameter related to $x$, e.g., the distance between $x$ to some fixed point $x_*$. Assume there exist positive constants $\epsilon_0, \alpha_1,\alpha_2,\beta,\overline{g}_{1},\underline{g}_{2}$, such that when $\epsilon\leq\epsilon_0$, we have
\begin{equation}\label{eq:limit_in_basiclemma}
|g_1(x)|\leq \overline{g}_{1}; \quad
|g_2(x)|\geq \underline{g}_{2}; \quad
|r_i(x)|\leq \alpha_i\epsilon^{\beta} \,\, \mbox{for $i=1,2$}.
\end{equation}
\begin{enumerate}[(i)]

\item Assume when $\epsilon\leq\epsilon_0$, there exists positive constants 
$\underline{g}_{1}$, $\overline{g}_{2}$, such that $f(x)$ and $g_1(x)$ further satisfy
\begin{equation}
\frac{g_1(x)}{g_2(x)}\geq0, \quad
\frac{g_1(x)+r_1(x)}{g_2(x)+r_2(x)}\geq0, \quad
|g_1(x)|\geq \underline{g}_{1},\quad
|g_2(x)|\leq \overline{g}_{2},
\end{equation}
then there exists positive constants $\epsilon_1$, $\kappa_1^{\rm (b)}$ depending on 
constants $\epsilon_0,\underline{g}_{1},\underline{g}_{2},\overline{g}_{1},\overline{g}_{2},\alpha_1,\alpha_2,\beta$ such that when $\epsilon\leq\epsilon_1$, 
\begin{equation}\label{eq:basiclemmaeqs0}
\left| \sqrt{\frac{g_1(x)+r_1(x)}{g_2(x)+r_2(x)}}-\sqrt{\frac{g_1(x)}{g_2(x)}}\right|\leq \kappa_1^{\rm (b)}\epsilon^{\beta}.
\end{equation}
		
\item Assume there exists functions $s_1(x)$ and $s_2(x)$ satisfying 
$|s_1(x)|\leq\alpha_3\epsilon^{2\beta}$ and 
$|s_2(x)|\leq\alpha_4\epsilon^{2\beta}$ when $\epsilon\leq\epsilon_0$.
Then there exists positive constants $\epsilon_2$, $\kappa_2^{\rm (b)}$ 
depending on $\epsilon_0,\overline{g}_{1},\underline{g}_{2},\alpha_1,\alpha_2,\alpha_3,\alpha_4,\beta$, such that when $\epsilon\leq\epsilon_2$, 
\[
\left|\frac{g_1+r_1+s_1}{g_2+r_2+s_2}-\frac{g_1}{g_2}-\frac{r_1}{g_2} +\frac{g_1r_2}{g_2^2}\right|\leq \kappa_2^{\rm (b)}\epsilon^{2\beta}.			
\]

\item  Let $s_1(x),s_2(x), \underline{g}_{1}$ be defined as 
in (i) and (ii). Assume $g_1(x)\geq0$ and $g_1(x)+r_1(x)+s_1(x)\geq0$ 
when $\epsilon\leq\epsilon_0$. Then there exists positive constants 
$\epsilon_3, \kappa_3^{\rm (b)}$ depending on $\epsilon_0,\underline{g}_{1},\alpha_1,\alpha_3,\beta$ such that when $\epsilon\leq\epsilon_3$,
\[
\left|\sqrt{g_1(x)+r_1(x)+s_1(x)}-\sqrt{g_1(x)}-\frac{r_1(x)}{2\sqrt{g_1(x)}}\right|\leq \kappa_3^{\rm (b)}\epsilon^{2\beta}.
\]
\end{enumerate}
\end{lemma}
\begin{proof}


For the first inequality, note that when 
$|g_1(x)|\geq|r_1(x)|$ and $|g_2(x)|\geq|r_2(x)|$, 
\[\left|\sqrt{\frac{g_1(x)+r_1(x)}{g_2(x)+r_2(x)}}-\sqrt{\frac{g_1(x)}{g_2(x)}}\right|\leq\max\left\{\sqrt{\frac{|g_1(x)|+|r_1(x)|}{|g_2(x)|-|r_2(x)|}}-\sqrt{\frac{|g_1(x)|}{|g_2(x)|}}, \sqrt{\frac{g_1(x)}{g_2(x)}}-\sqrt{\frac{|g_1(x)|-|r_1(x)|}{|g_2(x)|+|r_2(x)|}}\right\}.\]
Thus we only need to prove both $\sqrt{\frac{|g_1|+|r_1|}{|g_2|-|r_2|}}-\sqrt{\frac{|g_1|}{|g_2|}}$ and $ \sqrt{\frac{g_1}{g_2}}-\sqrt{\frac{|g_1|-|r_1|}{|g_2|+|r_2|}}$ can be controlled by a term of the form in \eqref{eq:basiclemmaeqs0}. We only prove the case for $\sqrt{\frac{|g_1|+|r_1|}{|g_2|-|r_2|}}-\sqrt{\frac{|g_1|}{|g_2|}}$. The proof for the second case is similar.
		
Assume $\epsilon\leq \min\left\{\left(\frac{\underline{g}_{1}}{2\alpha_1}\right)^{\frac{1}{\beta}}, \left(\frac{\underline{g}_{2}}{2\alpha_2}\right)^{\frac{1}{\beta}}\right\}$, we have $|g_1(x)|\geq|r_1(x)|$, $|g_2(x)|\geq|r_2(x)|$, and
\begin{equation}
\begin{aligned}
				\sqrt{\frac{|g_1|+|r_1|}{|g_2|-|r_2|}}-\sqrt{\frac{|g_1|}{|g_2|}}&\leq\sqrt{\frac{|g_1|+\alpha_1\epsilon^{\beta}}{|g_2|-\alpha_2\epsilon^{\beta}}}-\sqrt{\frac{|g_1|}{|g_2|}}.
\end{aligned}
\end{equation}
Let $t=\epsilon^{\beta}$ and consider 
$\widetilde{h}(t) = \sqrt{\frac{|g_1|+\alpha_1t}{|g_2|-\alpha_2t}}$. 
By the middle-value theorem, there exists $\xi\in (0,t)$, such that
\[
\widetilde{h}(t) = \widetilde{h}(0) + \widetilde{h}'(\xi)t
= \sqrt{\frac{|g_1|}{|g_2|}} + \frac{\alpha_1g_2+\alpha_2g_1}{2\sqrt{g_1+\alpha_1\xi}\left(g_2-\alpha_2\xi\right)^{\frac{3}{2}}}t.
\]
Since $\epsilon\leq \min\left\{\left(\frac{\underline{g}_{1}}{2\alpha_1}\right)^{\frac{1}{\beta}}, \left(\frac{\underline{g}_{2}}{2\alpha_2}\right)^{\frac{1}{\beta}}\right\}$, we have
\[
|g_1+\alpha_1\xi|\geq\frac{1}{2}\underline{g}_{1}, \quad  
|g_2-\alpha_2\xi|\geq\frac{1}{2}\underline{g}_{2},
\]
Thus
\[
\left|\widetilde{h}(t)-\sqrt{\frac{|g_1|}{|g_2|}}\right|\leq 2\frac{\alpha_1\overline{g}_{2}+\alpha_2\overline{g}_{1}}{\sqrt{\underline{g}_{1}\underline{g}_{2}^3}}t.\]
Using the definition of $t$ and $\widetilde{h}$, we obtain the result.

\medskip

For the second inequality, we only need to note that
		\[
		\begin{aligned}
			&\left|\frac{g_1+r_1+s_1}{g_2+r_2+s_2}-\frac{g_1}{g_2}-\frac{r_1}{g_2} +\frac{g_1r_2}{g_2^2}\right|\\
			\leq&\left|\frac{s_1}{g_2+r_2+s_2}\right|+\left|\frac{-g_1s_2-r_1r_2-r_1s_2}{(g_2+r_2+s_2)g_2}\right|+\left|\frac{g_1r_2^2+g_1r_2s_2}{(g_2+r_2+s_2)g_2^2}\right|.
		\end{aligned}  \]
		The result comes from straight calculation.

\medskip 

For the last inequality, 
consider $f(t)=\sqrt{g_1(x)+t}$ for $t>-g_1(x)$, which satisfies: 
\[
f(t) = 
\sqrt{g_1(x)}+\frac{1}{2\sqrt{g_1(x)}}t-\frac{1}{8}(g_1(x)+\xi)^{-\frac{3}{2}}t^2,
\]
where $\xi$ is between $0$ and $t$. 
Now let $t=r_1(x)+s_1(x)$. 
When $\epsilon\leq \min\left\{\epsilon_0, 1, \left(\frac{|g_{\rm inf,1}|}{4\alpha_1}\right)^{\frac{1}{\beta}}, \left(\frac{|g_{\rm inf,1}|}{4\alpha_3}\right)^{\frac{1}{2\beta}}\right\}$, $|t|\leq \frac{|g_{\rm \inf,1}|}{2}$. Thus we have
\[
\sqrt{g_1+r_1+s_1}
=\sqrt{g_1}+\frac{r_1+s_1}{2\sqrt{g_1}}-\frac{1}{8}(g_1+\xi)^{-\frac{3}{2}}(r_1+s_1)^2,\]
where $\xi$ is between $0$ and $r_1(x)+s_1(x)$. Therefore
\[
\left|\sqrt{g_1+r_1+s_1}-\sqrt{g_1}-\frac{r_1}{2\sqrt{g_1}}\right|\leq
\frac{\alpha_3\epsilon^{2\beta}}{2\sqrt{g_{\rm \inf,1}}}+\frac{\sqrt{2}}{4g_{\rm \inf,1}^{\frac{3}{2}}}(\alpha_1+\alpha_3)^2\epsilon^{2\beta}
\]
\end{proof}

\bibliographystyle{siamplain}
\bibliography{2devp}
\end{document}